\documentclass[12pt]{article}
\usepackage{mathtools}
\usepackage{amsthm}
\usepackage{csquotes}
\usepackage{amssymb}
\usepackage{graphicx}
\usepackage{subfig}
\usepackage{soul}  
\usepackage{xcolor}
\usepackage{empheq}
\usepackage{lipsum}
\usepackage{fullpage}
\usepackage{float} %Force figure placement in text
\usepackage[top= 2cm, bottom = 2 cm, left = 2.2 cm, right= 2.2 cm]{geometry}  
\usepackage[obeyspaces,hyphens,spaces]{url}
\usepackage[unicode]{hyperref}
\hypersetup{
    colorlinks=true,
    linkcolor=blue,
    citecolor=blue,
    filecolor=magenta,
    urlcolor=cyan
}
\usepackage{listings}

\definecolor{codegreen}{rgb}{0,0.6,0}
\definecolor{codegray}{rgb}{0.5,0.5,0.5}
\definecolor{codepurple}{rgb}{0.58,0,0.82}
\definecolor{backcolour}{rgb}{0.95,0.95,0.92}

\lstdefinestyle{mystyle}{
	backgroundcolor=\color{backcolour},   
	commentstyle=\color{codegreen},
	keywordstyle=\color{magenta},
	numberstyle=\footnotesize\color{codegray},
	stringstyle=\color{codepurple},
	basicstyle=\ttfamily\small,
	breakatwhitespace=false,         
	breaklines=true,                 
	captionpos=b,                    
	keepspaces=true,                 
	numbers=left,                    
	numbersep=5pt,                  
	showspaces=false,                
	showstringspaces=false,
	showtabs=false,                  
	tabsize=2
}

\definecolor{seagreen}{rgb}{0.18, 0.55, 0.34}
\definecolor{mediumviolet-red}{rgb}{0.78, 0.08, 0.52}
\definecolor{khaki}{rgb}{0.94, 0.9, 0.55}

\lstdefinelanguage{mypython}
{
	keywords=[1]{from, import, assert, not, print},
	keywordstyle=[1]{\color{mediumviolet-red}},
	keywords=[2]{surecr, torch, cp, lo, pl},
	keywordstyle=[2]{\color{seagreen}},
	numbers=none,
	upquote=true,
	showstringspaces=false,
	basicstyle=\ttfamily,
	columns=fullflexible,
	keepspaces=true,
	emph={True,False,as,def,return,float,class,match,switch,len},
	emphstyle={\color{seagreen}},
	frame=trBL,
	belowskip=1em,
	aboveskip=1em,
	captionpos=b
}

\usepackage[capitalize, nameinlink, noabbrev]{cleveref}
\crefname{equation}{}{}
\crefname{chapter}{Chapter}{Chapters}
\crefname{item}{item}{items}
\crefname{figure}{Figure}{Figures}
\crefname{theorem}{Theorem}{Theorems}
\crefname{lemma}{Lemma}{Lemmas}
\crefname{proposition}{Proposition}{Propositions}
\crefname{corollary}{Corollary}{Corollarys}
\crefname{definition}{Definition}{Definitions}
\crefname{fact}{Fact}{Facts}
\crefname{example}{Example}{Examples}
\crefname{algorithm}{Algorithm}{Algorithms}
\crefname{remark}{Remark}{Remarks}
\crefname{note}{Note}{Notes}
\crefname{notation}{Notation}{Notations}
\crefname{case}{Case}{Cases}
\crefname{exercise}{Exercise}{Exercises}
\crefname{question}{Question}{Questions}
\crefname{claim}{Claim}{Claims}
\crefname{enumi}{}{}

\numberwithin{equation}{section}

\usepackage{amsmath,xparse}
\makeatletter
\NewDocumentCommand{\lplabel}{o m}{%
	\makebox[0pt][r]{#2\hspace*{2em}}%
	\IfNoValueF{#1}
	{\def\@currentlabel{#2}\ltx@label{#1}}
}
\makeatother

\theoremstyle{plain}
\newtheorem{theorem}{Theorem}[section]

\newtheorem{fact}{Fact}[section]
\newtheorem{lemma}{Lemma}[section]

\theoremstyle{definition}

\newtheorem{example}{Example}[section]

\usepackage{enumitem}

\newcommand{\subjectto}{\ensuremath{\operatorname{subject~to}}}
\newcommand{\minimize}{\ensuremath{\operatorname{minimize}}}
\newcommand{\maximize}{\ensuremath{\operatorname{maximize}}}

\newcommand{\Pro}{\ensuremath{\operatorname{P}}}

\newcommand{\sign}{{\mbox{\bf sign}}} 
\newcommand{\symm}{{\mbox{\bf S}}} 

\providecommand{\abs}[1]{\left|#1\right|}
\providecommand{\norm}[1]{\left\lVert#1\right\rVert}
\providecommand{\innp}[1]{\left\langle#1\right\rangle}

\begin{document}

\title{Alternating Subgradient Methods for Convex-Concave Saddle-Point Problems}

\author{
	 Hui Ouyang\thanks{Department of Electrical Engineering, Stanford University}
}
\date{May 24, 2023}

\maketitle

\begin{abstract}
We propose   an alternating subgradient method with non-constant step sizes for
solving convex-concave saddle-point problems 
associated with general convex-concave functions.
We assume that the sequence of our step sizes is not summable but square summable. 
Then under the popular assumption of uniformly bounded subgradients, we prove
that a sequence of convex combinations of function values over our iterates
converges to the value of the function at a saddle-point. 
Additionally, 
based on our result regarding the boundedness of the sequence of our iterates,
we show that a sequence of the function evaluated at convex combinations of 
our iterates  also converges to the value of the function over a saddle-point.

We implement   our algorithms in  examples of 
a linear program in inequality form, 
a least-squares problem with $\ell_{1}$ regularization,
a matrix game, and a robust Markowitz portfolio construction problem.
To accelerate convergence, we reorder  the sequence of step sizes
in descending order, which turned out to work very-well  in our examples.
Our convergence results are confirmed by our numerical experiments.
Moreover, we also numerically compare our iterate scheme with iterates schemes
associated with constant step sizes. 
Our numerical results support  our choice of step sizes. 
Additionally,  we observe the convergence of the sequence of function values
over our iterates in multiple experiments, 
which currently lacks theoretical support.
 
\end{abstract}

%\newpage
%\tableofcontents
%\newpage

 \section{Introduction}
In the whole work,  
 $\mathcal{H}_{1}$ and $\mathcal{H}_{2}$ are Hilbert spaces and   $X \subseteq 
 \mathcal{H}_{1}$ and $Y \subseteq \mathcal{H}_{2}$ 
 are nonempty closed and convex sets. 
 The Hilbert direct sum $\mathcal{H}_{1} \times \mathcal{H}_{2}$ of  
 $\mathcal{H}_{1}$ and $\mathcal{H}_{2}$ is equipped with the inner product 
\begin{align}
 \label{eq:innerproduct}
 	(\forall (x, y) \in X \times Y ) (\forall (u,v) \in X \times Y) \quad \innp{(x,y), (u,v)} = 
 	\innp{x,u} +\innp{y,v}
\end{align}
 and the induced norm
\begin{align}\label{eq:norm}
 	(\forall (x, y) \in X \times Y ) \quad \norm{(x, y) }^{2} =\innp{(x,y), (x,y)}=\innp{x,x} 
 	+\innp{y,y} =\norm{x}^{2} + \norm{y}^{2}. 
\end{align}
 
 Throughout this work, $f : X \times Y \to \mathbf{R} \cup \{ - \infty, + \infty \}$ satisfies 
 that 
 $(\forall y \in Y)$ $f(\cdot, y) : X \to \mathbf{R}  \cup \{ + \infty\}$ is proper and convex, 
 and  that 
 $(\forall x \in X)$ $f(x, \cdot) : Y \to \mathbf{R}  \cup \{ - \infty\}$ is proper and concave. 
 (It is referred to as \emph{convex-concave function} from now on.)
 
In this work, we aim to solve the following \emph{convex-concave saddle-point 
problem}
\begin{align}\label{eq:problem}
 	\underset{x \in X}{\minimize} ~\,	\underset{y \in Y}{\maximize}  ~f(x,y).
\end{align}
We assume that $(\forall (\bar{x}, \bar{y}) \in X \times Y)$ 
 $\partial_{x}f(\bar{x}, \bar{y}) \neq \varnothing$ and 
 $\partial_{y} (- f(\bar{x},  \bar{y})) \neq \varnothing$. 
 We also suppose that  the solution set of \cref{eq:problem} is nonempty, 
 that is, there exists a \emph{saddle-point} $(x^{*}, y^{*}) \in X \times Y$ of $f$
 satisfying 
\begin{align}\label{eq:solution}
(\forall x \in X) (\forall y \in Y) \quad f(x^{*}, y)  \leq  f(x^{*}, y^{*})  \leq f(x, y^{*}).
\end{align}

 \subsection{Related work}
 Convex-concave saddle-point problems arise in a wide range of applications such 
 as resource allocation problems for networked-systems, game theory, 
 finance, robust and minimax optimization, image processing, 
 generative adversarial networks, adversarial training, robust optimization, 
 primal-dual reinforcement learning, and   machine learning   
 (see, e.g., 
 \cite{NedicOzdaglar2009}, 
 \cite{BV2004}, \cite{SLB2023}, \cite{ChambollePock2011},  
 \cite{WandLi2020improved}, and \cite{ZhangWangLessardGrosse2022}  for 
 details).
 Based on the easy implementation, low memory requirement, 
 and  low barrier for  usage, 
 the subgradient method is one of the most popular methods for solving 
convex-concave saddle-point problems.
 
 For interested readers, 
 we recommend \cite[Section~1]{NedicOzdaglar2009} with a summary of various 
 subgradient methods for solving convex-concave saddle-point problems 
and \cite[Section~3]{WandLi2020improved} with a list of different
algorithms and literature on solving special cases or more general versions of 
convex-concave saddle-point problems.
Among the literature on subgradient methods for solving convex-concave
saddle-point problems,
the author in \cite{Nesterov2009} proposed a primal-dual subgradient method with 
two control sequences 
(one aggregates the support functions in the dual space, and the other 
establishes a dynamically updated scale between the primal and dual spaces) 
for different types of nonsmooth problems with convex structure. 
Moreover, the author provides a variant of the proposed subgradient scheme
for convex-concave saddle-point problems in  \cite[Section~4]{Nesterov2009},
and shows an upper bound on a sequence constructed by their iterates
 under some assumptions, including the existence of certain strongly convex 
 prox-functions and the uniform boundedness of subgradients.
In addition, the authors of  \cite{SLB2023} worked on saddle problems
including the partial supremum or infimum of convex-concave functions,
which is a more general version of the convex-concave saddle-point problem
considered in this work.
They applied the language and methods presented in 
\cite{JuditskyNemirovski2022},
 used the idea of conic-representable saddle-point programs to automatically 
 reduce a saddle problem to a single convex optimization problem,
and developed an open-source package called DSP for
users to easily formulate and solve saddle problems. 

\subsection{Comparison with related work} 
Main theoretical results in this work are mainly inspired by  
\cite[Section~3]{NedicOzdaglar2009} by Nedi\'{c} and  Ozdaglar and 
\cite[Section~6]{BoydSGMNotes} by Boyd. 
In \cite[Section~3]{NedicOzdaglar2009}, the authors worked on a subgradient 
algorithm with an averaging scheme for generating approximate saddle-points of 
a convex-concave function. They also showed the convergence of function values 
at iterate averages to the the function value at a saddle-point, 
with the error level being a function of the step-size value,
under some assumptions regarding the boundedness of the sequence of iterates or 
compactness of related sets.  
\cite{BoydSGMNotes} is a lecture note by Boyd for  the course EE364b in Stanford 
University, which covers various subgradient methods and techniques of their 
convergence proofs. 
Convex-concave saddle-point problems are not considered 
in \cite{BoydSGMNotes} yet. 
Here, we apply some convergence proofs techniques
and popular assumptions on sequences of step sizes,
as presented in \cite{BoydSGMNotes}. 
We state main differences between this work and \cite{NedicOzdaglar2009} 
below.
\begin{itemize}
	\item Our iterate scheme \cref{eq:algorithm} replaces the constant 
	step-size $\alpha$ in the iterate scheme worked in \cite{NedicOzdaglar2009} 
	by a sequence $(t_{k})_{k \in 	\mathbf{N}}$.
	\item Under some  boundedness of the sequence of iterates or compactness of 
	related sets, 
	authors in \cite[Section~3]{NedicOzdaglar2009} considered   the 
	convergence of $\frac{1}{ k+1}	\sum^{k}_{i =0}  f(x^{i}, y^{i}) \to f(x^{*}, y^{*}) $  
	and 
	$f(\tfrac{1}{k}\sum^{k-1}_{i=0} x^{i}, \tfrac{1}{k}\sum^{k-1}_{i=0} y^{i}) \to 
	f(x^{*},y^{*}) $,  
	where $(x^{*},y^{*}) $ is a saddle-point of $f$ and  $\left((x^{k}, y^{k})\right)_{k 
		\in 	\mathbf{N}}$    
	is generated by the iterate scheme \cref{eq:algorithm} below with $(\forall k \in 
	\mathbf{N})$ $t_{k} 
	\equiv \alpha \in \mathbf{R}_{++}$. 
	
	In this work, we directly establish the boundedness of the sequence 
	$((x^{k}, y^{k}))_{k \in \mathbf{N}}$
	of iterates generated by our scheme \cref{eq:algorithm}. 
	So without any assumption  on the boundedness of the sequence of iterates or 
	compactness  of related sets, 
	we prove $	\sum^{k}_{i =0} \frac{t_{i}}{ \sum^{k}_{j=0} t_{j}} f(x^{i}, y^{i}) \to  
	f(x^{*}, y^{*}) $ in \cref{theorem:sumconverge}
	and $f\left( \sum^{k}_{i =0} \frac{t_{i}}{ \sum^{k}_{j=0} t_{j}} x^{i} , 
	\sum^{k}_{i =0} \frac{t_{i}}{ \sum^{k}_{j=0} t_{j}} y^{i}\right) \to  f(x^{*}, y^{*})$ in 
	\cref{theorem:xhatkyhatkConverge}.
	\item In  \cite{NedicOzdaglar2009}, after section~3, by replacing the second	
	projector in their original iterate scheme with a projector onto a compact convex set 
	containing the set of dual optimal solutions, 
	the authors introduced their primal-dual subgradient method;
	moreover, under some standard Slater constraint qualification and
	uniformly boundedness assumption of related subgradients,
	 the authors theoretically  studied the estimate on the convergence 
	rate of generated primal sequences for finding approximate primal-dual optimal 
	solutions as approximate saddle points of the Lagrangian function of a convex 
	constrained optimization problem. 
	
	Although we don't theoretically work on finding saddle-points of
	Lagrangian  (which are special cases of convex-concave functions), 
	we numerically apply our algorithm to  Lagrangian functions 
	of linear program in inequality form 
	and of least-squares problem with $\ell_{1}$ regularization in
	\cref{subsection:LP,subsection:LSl1}. 
	Furthermore,  in \cref{subsection:matrixgame,subsection:rmpcp}, 
	we also numerically confirm our theoretical results with problems of
	matrix game and robust Markowitz portfolio construction problem.
	In addition, we also apply our algorithm to
	 another Lagrangian of an easy constrained convex problem  in 
	 \cref{subsection:toyexample}
	to show the importance or benefit of replacing the constant step-size $\alpha$  
	with a sequence $(t_{k})_{k \in 	\mathbf{N}}$ 
	satisfying some popular constraints of step sizes.
\end{itemize}

  \subsection{Outline}
 
 We mainly present some examples of convex-concave saddle-point problems
 in \cref{section:Preliminaries}. 
 In \cref{section:ASM}, we introduce our alternating subgradient methods
 for convex-concave saddle-point problems associated with
 a general convex-concave function. 
 Then in the same section, we prove two desired results on
 the convergence to the value of function over a saddle-point. 
 In \cref{section:NumericalExperiments}, 
 we implement our algorithm and compare it to some other 
 schemes of iterates in examples of the following problems:
 linear program in inequality form, 
 least-squares problem with $\ell_{1}$ regularization,
 matrix game, and robust Markowitz portfolio construction problem.
 To accelerate convergence, in our numerical experiments,
 we reorder the sequence $(t_{k})_{k \in \mathbf{N}}$ of step sizes
 in descending order, which turns out to work very-well  in our examples.
We  sum up our work in \cref{section:conclusion}.

 \section{Preliminaries} \label{section:Preliminaries}
We point out some notation used in this work below.  
$\mathbf{R}$, $\mathbf{R}_{+}$, $\mathbf{R}_{++}$, and $\mathbf{N}$  are the set 
of all real numbers, the set of all nonnegative real numbers, 
the set of all positive real numbers, 
and the set of all nonnegative integers, respectively. 
Let $C$ be a nonempty  closed and convex subset of $\mathcal{H}$.  
The \emph{projector} (or \emph{projection operator}) onto $C$ is the operator, 
denoted by $\Pro_{C}$,  that maps every point in $\mathcal{H}$ to its unique 
projection onto $C$, 
that is, $(\forall x \in \mathcal{H})$ $\norm{x - \Pro_{C}x} = \inf_{c \in C} 
\norm{x-c}$.  
Let $\mathcal{H}$ be a Hilbert space and let $g: \mathcal{H} \to \left]-\infty, 
+\infty\right]$ be proper. The \emph{subdifferential of $g$} is the set-valued operator 
\[ 
	\partial g: \mathcal{H} \to 2^{\mathcal{H}}: x \mapsto \{ u \in \mathcal{H} ~:~ 
	(\forall y 
	\in \mathcal{H})~ \innp{u, y-x} + f(x) \leq f(y) \}.
\]

The following result is well-known and will be used several times later. For 
completeness, we show some details below. 
\begin{fact} \label{fact:saddlepoint}
	Let $(\bar{x}, \bar{y}) \in X \times Y$. The following statements are equivalent. 
	\begin{enumerate}
		\item  \label{fact:saddlepoint:sp} $(\bar{x}, \bar{y})$ is a saddle-point of the 
		function $f$.
		\item  \label{fact:saddlepoint:leq} $(\forall (x,y) \in X \times Y)$ $f(\bar{x},y)- 
		f(x,\bar{y}) \leq 0$.
		\item  \label{fact:saddlepoint:partial} $0 \in \partial_{x} f(\bar{x}, \bar{y}) $ and 
		$0 
		\in \partial_{y} (-f(\bar{x}, \bar{y})) $.
	\end{enumerate}
\end{fact}

\begin{proof}
	\cref{fact:saddlepoint:sp}  $\Leftrightarrow$ \cref{fact:saddlepoint:leq}:  This is 
	trivial 
	by recalling the definition \cref{eq:solution} of the saddle-point.
	
	\cref{fact:saddlepoint:sp}  $\Leftrightarrow$ \cref{fact:saddlepoint:partial}: 
	According 
	to the definition of saddle-point \cref{eq:solution}, we observe that 
	\cref{fact:saddlepoint:sp}   is equivalent to the following
\[
		(\forall x \in X)~ \innp{0, x-\bar{x}} + f(\bar{x}, \bar{y})  \leq  f(x, \bar{y}) \text{ 
		and } 
		(\forall y \in Y) ~\innp{0, y-\bar{y}} + (-f(\bar{x}, \bar{y})) \leq  -f(\bar{x}, y),
 \]
	which is exactly \cref{fact:saddlepoint:partial}.
\end{proof}

We present some examples of saddle-points below.
\begin{example}  \label{example:toy}
	Let $(\forall i \in \{1,2\})$ $f_{i}: \mathcal{H}_{i} \to \mathbf{R}$ be proper, convex, 
	and differentiable function. Define $f: \mathcal{H}_{1} \times \mathcal{H}_{2} \to 
	\mathbf{R}$ as
 $(\forall (x,y) \in \mathcal{H}_{1} \times \mathcal{H}_{2})$ $f(x,y) = f_{1}(x) 
 -f_{2}(y)$. 
 Let  $(\bar{x}, \bar{y}) \in  \mathcal{H}_{1} \times \mathcal{H}_{2}$.  Then, by 
 \cref{fact:saddlepoint},  $(\bar{x}, \bar{y})$ is a saddle-point of the function $f$ if 
 and 
 only if $0=\nabla f_{1}(\bar{x})$ and $0=\nabla f_{2}(\bar{y})$.   
\end{example}
 
 We shall revisit examples below in \cref{section:NumericalExperiments} 
 on numerical experiments. 
\begin{example}[Lagrange duality]  \label{example:lagrange}
Note that when strong duality holds,
a pair of primal and dual optimal points turns out to be a saddle-point;
and  the value of the Lagrangian over a saddle-point equals to 
both the optimal solutions of corresponding primal and dual problems. 
We display Lagrangian of some problems for which strong duality obtains below. 
		\begin{enumerate}
		\item  \label{example:lagrange:toy} \cite[Exercise~5.1]{BV2004}  
		Consider the primal problem
\begin{align*}
			&\minimize~x^2 +1 \\
		\lplabel{(P$_{1}$)}	&\subjectto~(x-2)(x-4) \leq 0,		
\end{align*}
with variable $x \in \mathbf{R}$. 
The corresponding Lagrangian $L : \mathbf{R} \times \mathbf{R}_{+}  \to \mathbf{R}$ is 
\[
		\lplabel{(L$_{1}$)}	 
		L(x,y) = x^2 +1  +y (x-2)(x-4) = x^2 (1+y) - 6xy +8y+1,
 \]
	and the dual problem is 
\begin{align*}
		&\maximize~10-y - \frac{9}{y+1}\\
		\lplabel{(D$_{1}$)}	&\subjectto~y \geq 0,
\end{align*}
with variable $y \in \mathbf{R}$.
It is easy to know that 
the primal and dual optimal points are $x^{*} =2$ and $y^{*} =2$, respectively, 
and that the primal and dual optimal values are $p^{*}=d^{*} =5$. 
Then we deduce that $(x^{*}, y^{*}) =(2,2)$ is a saddle-point of the function $L(x,y)$ 
and that 
\[ 
\sup_{y \in \mathbf{R}_{+}} \inf_{x \in \mathbf{R}} L(x,y) = L(2,2)=5 
= \inf_{x \in \mathbf{R}} \sup_{y \in \mathbf{R}_{+}}L(x,y).
\]
		\item  \label{example:lagrange:lp}
		\cite[Section~5.2.1]{BV2004} Consider the inequality form LP
\begin{align*}
			&\minimize~c^{T}x \\
			\lplabel{(P$_{2}$)}	&\subjectto~Ax \leq b,		
\end{align*}
	where $x \in\mathbf{R}^{n}$ is the variable, and  $A \in \mathbf{R}^{m \times n}$, $c 
	\in 
	\mathbf{R}^{n}$, and $b  \in \mathbf{R}^{m}$.
As it is stated on \cite[Page~225]{BV2004}, 
the  Lagrangian $L : \mathbf{R}^{n}  \times \mathbf{R}^{m}_{+}  \to \mathbf{R}$ is 
\[  
\lplabel{(L$_{2}$)}	 
L(x,y) = c^{T}x + y^{T}(Ax-b) = -b^{T}y + (A^{T} y+c)^{T}x = y^{T}Ax +c^{T}x   -b^{T}y,
\]
and the dual problem is 
\begin{align*}
			&\maximize~-b^{T}y\\
			\lplabel{(D$_{2}$)}	&\subjectto~y \geq 0\\
			&\quad \quad \quad \quad \quad A^{T} y+c =0,
\end{align*}
with variable $y \in \mathbf{R}^{m}$.
Suppose that $x^{*}$ and  $ y^{*}$ are the optimal points of (P$_{2}$) and 
(D$_{2}$), respectively. 
It is clear that Slater's condition holds, so
$(x^{*}, y^{*})$ is a saddle-point of the function $L(x,y)$ and that 
\[  
\sup_{y \in \mathbf{R}^{m}_{+}} \inf_{x \in \mathbf{R}^{n}} L(x,y) = L(x^{*}, y^{*} )
= \inf_{x \in \mathbf{R}^{n}} \sup_{y \in \mathbf{R}^{m}_{+}}L(x,y).
\]
		\item  \label{example:lagrange:lsl1}
\cite[Exercise~5.34(b)]{BV2004Additional} Consider the following least-squares 
problem with $\ell_{1}$ regularization
\[  
 \minimize~\frac{1}{2} \norm{Ax -b}^{2}_{2} + \gamma \norm{x}_{1}
\]
where $x \in \mathbf{R}^{n}$ is the variable, and $A \in \mathbf{R}^{m \times n}$, $b  
\in \mathbf{R}^{m}$, 
and $\gamma \in  \mathbf{R}_{++}$.
Clearly, it is equivalent to 
\begin{align*}
&\minimize~\frac{1}{2} \norm{u}^{2}_{2} + \gamma \norm{x}_{1}\\
\lplabel{(P$_{3}$)}	&\subjectto~u=Ax -b.
\end{align*}
The corresponding Lagrangian $L :\mathbf{R}^{n} \times \mathbf{R}^{m}  \times 
\mathbf{R}^{m}   \to \mathbf{R}$ is 
\begin{align*}
\lplabel{(L$_{3}$)}	 
L(x,u,y) =& \frac{1}{2} \norm{u}^{2}_{2} + \gamma \norm{x}_{1} + y^{T}( Ax 
-b -u)\\
 =& y^{T} \begin{pmatrix}
				A &-I
			\end{pmatrix} 
		\begin{pmatrix}
			x \\u
		\end{pmatrix} 
	+\frac{1}{2} \norm{u}^{2}_{2} + \gamma \norm{x}_{1} - y^{T}b.
\end{align*}
	Note that 
\begin{align*}
		\inf_{x  \in \mathbf{R}^{n},u  \in \mathbf{R}^{m}}  L(x,u,y) 
		 =& \inf_{x  \in \mathbf{R}^{n},u  \in 
		\mathbf{R}^{m}}  \frac{1}{2} \norm{u}^{2}_{2} + \gamma \norm{x}_{1} + y^{T}( Ax 
		-b 
		-u) \\
		 =& -b^{T}y + \inf_{x \in \mathbf{R}^{n}} \left( (A^{T}y)^{T}x + \gamma 
		\norm{x}_{1} 
		\right) +  \inf_{u \in \mathbf{R}^{m}} \left( -y^{T}u + \frac{1}{2} \norm{u}^{2}_{2} 
		\right), 
\end{align*}
that, by  \cite[Section~5.1.6 and Example~3.26]{BV2004}, 
\begin{align*}
	\inf_{x \in \mathbf{R}^{n}}  (A^{T}y)^{T}x + \gamma \norm{x}_{1}     =& - \gamma 
	\sup_{x \in \mathbf{R}^{n}} -\frac{1}{\gamma}(A^{T}y)^{T}x  - \norm{x}_{1}\\
	 =& \begin{cases}
		0 \quad & \text{if } \norm{-\frac{1}{\gamma} A^{T}y}_{\infty} \leq 1\\
		-\infty \quad & \text{otherwise},
	\end{cases}
\end{align*}
and that
\[ 
\inf_{u \in \mathbf{R}^{m}}   -y^{T}u + \frac{1}{2} \norm{u}^{2}_{2} 
= \frac{1}{2} \inf_{u \in \mathbf{R}^{m}}     \norm{u-y}^{2}_{2} - \frac{1}{2} 
\norm{y}^{2}_{2} 
= - \frac{1}{2} \norm{y}^{2}_{2}.    
\]
	We derive that
\[  
		\inf_{x  \in \mathbf{R}^{n},u  \in \mathbf{R}^{m}}  L(x,u,y)  = -b^{T}y + \frac{1}{2} 
		\norm{y}^{2}_{2},
\]
with $ \norm{A^{T}y}_{\infty} \leq \gamma$,
		and that the dual problem is 
\begin{align*}
			&\maximize~-b^{T}y + \frac{1}{2} \norm{y}^{2}_{2}\\
			\lplabel{(D$_{3}$)}	&\subjectto~\norm{A^{T}y}_{\infty} \leq \gamma,
\end{align*}
with variable $y \in \mathbf{R}^{m}$. Suppose that $(x^{*},u^{*})$ and  $ y^{*}$ are 
the optimal points of (P$_{3}$) and (D$_{3}$), respectively. In view of Slater's theorem,  
$((x^{*},u^{*}), y^{*})$ is s a saddle-point of the function $L((x,u),y)$ and that 
\[
			\sup_{y \in \mathbf{R}^{m}} \inf_{x \in \mathbf{R}^{n},u  \in \mathbf{R}^{m}} 
			L((x,u),y) = 
			L((x^{*},u^{*}), y^{*})= \inf_{x \in \mathbf{R}^{n} ,u  \in \mathbf{R}^{m}} \sup_{y 
			\in 
			\mathbf{R}^{m}}L((x,u),y).
\]
	\end{enumerate}
\end{example}

\section{Alternating subgradient methods} \label{section:ASM}
Results in this section are inspired by \cite[Section~3]{NedicOzdaglar2009} 
by Nedi\'{c} and  Ozdaglar and  \cite{BoydSGMNotes} by Boyd 
for the course EE364b in Stanford University. 
In particular, the technique in the proof of \cref{theorem:sumconverge} mimics 
that in \cite[Section~3]{NedicOzdaglar2009}.  
The authors in \cite[Section~3]{NedicOzdaglar2009} considered   the 
 convergence of 
 $\frac{1}{ k+1}	\sum^{k}_{i =0}  f(x^{i}, y^{i}) \to f(x^{*}, y^{*}) $  and 
 $f(\tfrac{1}{k}\sum^{k-1}_{i=0} x^{i}, \tfrac{1}{k}\sum^{k-1}_{i=0} y^{i}) \to 
 f(x^{*},y^{*}) $ 
 under some assumptions including boundedness or compactness, 
 where $(x^{*},y^{*}) $ is a saddle-point of $f$ and  
 $\left((x^{k}, y^{k})\right)_{k \in 
 \mathbf{N}}$    
 is generated by \cref{eq:algorithm} below with $(\forall k \in \mathbf{N})$ $t_{k} 
 \equiv \alpha \in \mathbf{R}_{++}$.
In addition, the proof of the boundedness of the iteration sequence $\left((x^{k}, 
y^{k})\right)_{k \in \mathbf{N}}$  in \cref{lemma:xkykbounded}
 is motivated from \cite[Section~6]{BoydSGMNotes}, 
 which is essential to our convergence result in \cref{theorem:xhatkyhatkConverge}.

\subsection{Algorithm} 
Let $(x^{0}, y^{0}) \in X \times Y$. We consider the sequence of iterations generated by
\begin{align}\label{eq:algorithm}
	 (\forall k \in \mathbf{N}) \quad 
	 x^{k+1} = \Pro_{X}(x^{k} - t_{k} g_{k})  \quad \text{and} \quad y^{k+1} 
	 = \Pro_{Y}(y^{k} - t_{k} h_{k}),  
\end{align}
where $\Pro_{X}$ and $\Pro_{Y}$ are projectors onto $X$ and $Y$, respectively,   
$ (\forall k \in \mathbf{N})$  $g_{k} \in \partial_{x}f(x^{k},y^{k}) $ 
and $h_{k} \in \partial_{y}(- f(x^{k},y^{k}) )$, 
and $(\forall k \in \mathbf{N} \smallsetminus \{0\})$ $t_{k} \in \mathbf{R}_{+}$ 
and $t_{0} \in \mathbf{R}_{++}$.

Let's revisit the Lagrangian function of 
a least-squares problem with $\ell_{1}$ regularization,
presented in \cref{example:lagrange}\cref{example:lagrange:lsl1}, 
and illustrate the algorithm \cref{eq:algorithm} on this particular function.
\begin{example}  \label{example:lsl1:subgradient}
Consider the convex-concave function $f: \mathbf{R}^{n+m} \times \mathbf{R}^{m} \to 
\mathbf{R}$ 
defined as 
\[ 
(\forall ((x,u),y) \in \mathbf{R}^{n+m} \times \mathbf{R}^{m} )	\quad  f((x,u),y) 
= \frac{1}{2} \norm{u}^{2}_{2} + \gamma \norm{x}_{1} + y^{T}( Ax -b -u).
\] 
In view of \cite[Section~3.4]{BoydSubgradientsNotes},
\[ 
	(\forall x \in \mathbf{R}^{n}) \quad \sign(x) \in \partial \norm{x}_{1}, \quad 
	\text{where} \quad (\forall i \in \{1,\ldots,n\}) ~(\sign(x))_{i} = \begin{cases}
		1 \quad &\text{if } x_{i} >0\\
		0 \quad &\text{if } x_{i} =0\\
		-1 \quad &\text{if } x_{i} <0.
	\end{cases}
\]
Therefore, it is easy to see that in this case, the algorithm \cref{eq:algorithm} is 
simply 
that $((x_{0}, Ax_{0}-b), y_{0}))\in \mathbf{R}^{n+m} \times \mathbf{R}^{m}$ and for 
every 
$k \in 
\mathbf{N}$,
\begin{align*}
	(x^{k+1}, u^{k+1})  =& (x^{k},u^{k}) - t_{k} \left(A^{T}y^{k} + \gamma 
	\sign(x^{k}), 
	-y^{k}+u^{k}\right)\\
	  y^{k+1}  =& y^{k} -t_{k} \left(-Ax^{k}+u^{k}+b \right).
\end{align*}
\end{example}

\subsection{Convergence results}
Henceforth, $((x^{k},y^{k}))_{k \in \mathbf{N}}$ is constructed by our 
iterate scheme \cref{eq:algorithm}; moreover, 
\[
(\forall k \in \mathbf{N}) \quad  
\hat{x}_{k} = \sum^{k}_{i =0} \frac{t_{i}}{ \sum^{k}_{j=0} t_{j}} x^{i} 
\text{ and }
  \hat{y}_{k} = \sum^{k}_{i =0} \frac{t_{i}}{ \sum^{k}_{j=0} t_{j}} y^{i}, 
\]
where $(\forall k \in \mathbf{N} \smallsetminus \{0\})$ $t_{k} \in \mathbf{R}_{+}$ 
and $t_{0} \in \mathbf{R}_{++}$ are step sizes used in our iterate 
scheme \cref{eq:algorithm}.

Suppose $(x^{*}, y^{*}) \in X \times Y$  is a saddle point of 
a convex-concave function $f$.
In this subsection, we present convergence results 
\cref{theorem:sumconverge,theorem:xhatkyhatkConverge}
on $\sum^{k}_{i =0} \frac{t_{i}}{ \sum^{k}_{j=0} t_{j}} f(x^{i}, y^{i}) 
\to  f(x^{*}, y^{*}) $
and $f( \hat{x}_{k} ,\hat{y}_{k}) \to  f(x^{*}, y^{*})$, respectively. 
To this end, we need the following auxiliary results. 
\begin{lemma} \label{lemma:xkyk}
Let $x \in X$ and $y \in Y$. The following results hold. 
\begin{enumerate}
\item \label{lemma:xkyk:tk} For every $k \in \mathbf{N}$,
\begin{subequations}\label{eq:ineqtk}
\begin{align}
&t_{k} (f(x^{k}, y^{k}) - f(x,y^{k})) \leq 
\tfrac{1}{2} \left(\norm{x^{k} -x}^{2} - \norm{x^{k+1} -x}^{2} \right) 
+ \tfrac{1}{2} t_{k}^{2} \norm{g_{k}}^{2}; \label{eq:ineqtk:x}\\
&t_{k} (f(x^{k}, y) - f(x^{k},y^{k})) \leq \tfrac{1}{2} \left(\norm{y^{k} -y}^{2}
 - \norm{y^{k+1} -y}^{2} \right) + \tfrac{1}{2} t_{k}^{2} \norm{h_{k}}^{2}. 
 \label{eq:ineqtk:y}
\end{align}
\end{subequations}
\item \label{lemma:xkyk:sum} 
For every $k \in \mathbf{N}$,
\begin{subequations} \label{eq:ineqhatxkhatyk}
\begin{align}
&  \sum^{k}_{i =0} \frac{t_{i}}{ \sum^{k}_{j=0} t_{j}} f(x^{i}, y^{i}) -  f(x,\hat{y}_{k} )
 \leq \frac{ \norm{x^{0} -x}^{2} - \norm{x^{k+1} -x}^{2} }{2 \sum^{k}_{j=0} t_{j}} 
 + \frac{ \sum^{k}_{i =0}  t_{i}^{2} \norm{g_{i}}^{2} }{2 \sum^{k}_{j=0} t_{j} };
  \label{eq:ineqhatxkhatyk:x}\\
& f( \hat{x}_{k} ,y) -   \sum^{k}_{i =0} \frac{t_{i}}{ \sum^{k}_{j=0} t_{j}} f(x^{i}, y^{i}) 
\leq  \frac{ \norm{y^{0} -y}^{2} - \norm{y^{k+1} -y}^{2} }{2 \sum^{k}_{j=0} t_{j}}
+ \frac{ \sum^{k}_{i =0}  t_{i}^{2} \norm{h_{i}}^{2} }{2 \sum^{k}_{j=0} t_{j} }. 
\label{eq:ineqhatxkhatyk:y}
\end{align}
\end{subequations}
\end{enumerate}	
\end{lemma}

\begin{proof}
\cref{lemma:xkyk:tk}:	
Let $k \in \mathbf{N}$. Because $g_{k} \in \partial_{x}f(x^{k},y^{k}) $ and 
$h_{k} \in \partial_{y}(- f(x^{k},y^{k}) )$, we have that 
	\begin{subequations}
		\begin{align}
			&(\forall x \in X) \quad \innp{g_{k}, x -x^{k}}  
			\leq f(x,y^{k}) -f(x^{k},y^{k}) \label{eq:gkineq}; \\
		 &(\forall y \in Y) \quad \innp{h_{k}, y -y^{k}} \leq -f(x^{k},y) + f(x^{k},y^{k})  
		 \label{eq:hkineq}.
		\end{align}
	\end{subequations}

Note that $x =\Pro_{X} x$ and $\Pro_{X}$ is nonexpansive. 
According to \cref{eq:algorithm}, 
\begin{align*}
		\norm{x^{k+1} -x}^{2}  =& \norm{\Pro_{X}(x^{k} - t_{k} g_{k})   - 
		\Pro_{X}x}^{2}\\
		 \leq&  \norm{ x^{k} - t_{k} g_{k}    -x}^{2}\\
		 =&  \norm{ x^{k}  -x}^{2} -2t_{k} \innp{g_{k} , x^{k}  -x } +t_{k}^{2} \norm{ 
		g_{k}}^{2}  \\
		 \stackrel{\cref{eq:gkineq}}{\leq}&  \norm{ x^{k}  -x}^{2} -2t_{k} (f(x^{k},y^{k}) 
		- f(x,y^{k}) )+t_{k}^{2} \norm{ g_{k}}^{2},
\end{align*}
which implies \cref{eq:ineqtk:x}.
	
Similarly, employing  \cref{eq:algorithm} again and $y =\Pro_{Y} y$ in the first 
equality below
and applying   the fact that $\Pro_{Y}$ is nonexpansive in the first inequality 
below, 
we have that 
\begin{align*}
\norm{y^{k+1} -y}^{2} &=& \norm{\Pro_{Y}(y^{k} - t_{k} h_{k})   - \Pro_{Y} y}^{2}\\
&\leq&  \norm{ y^{k} - t_{k} h_{k}    -y}^{2}\\
&=&  \norm{ y^{k}  -y}^{2} -2t_{k} \innp{h_{k} , y^{k}  -y } +t_{k}^{2} \norm{ 
h_{k}}^{2}  \\
&\stackrel{\cref{eq:hkineq}}{\leq}&  \norm{ y^{k}  -y}^{2} -2t_{k} (f(x^{k},y) - 
f(x^{k},y^{k}) )+t_{k}^{2} \norm{ h_{k}}^{2},
\end{align*}
which yields  \cref{eq:ineqtk:y}.
	
\cref{lemma:xkyk:sum}:  
Let $k \in \mathbf{N}$. Because $f(x,\cdot)$ is concave and $f(\cdot, y)$ is convex, 
we observe that
\begin{subequations}
\begin{align} 
&\sum^{k}_{i =0} \frac{t_{i}}{ \sum^{k}_{j=0} t_{j}} f(x, y^{i}) 
\leq f\left( x,  \sum^{k}_{i =0} \frac{t_{i}}{ \sum^{k}_{j=0} t_{j}} y^{i} \right) = f(x, \hat{y}_{k}); \label{eq:fconcaveyhatyk}\\
& \sum^{k}_{i =0} \frac{t_{i}}{ \sum^{k}_{j=0} t_{j}} f(x^{i}, y) 
\geq f\left( \sum^{k}_{i =0} \frac{t_{i}}{ \sum^{k}_{j=0} t_{j}} x^{i} , y \right) 
= f( \hat{x}_{k}, y). \label{eq:fconvexxhatxk}
\end{align}
\end{subequations}

Sum \cref{eq:ineqtk:x} over $i \in \{0,1, \ldots, k\}$ to get that
\[  
\sum^{k}_{i=0} t_{i} (f(x^{i}, y^{i}) - f(x,y^{i})) \leq \tfrac{1}{2} (\norm{x^{0} -x}^{2} 
- \norm{x^{k+1} -x}^{2} ) + \tfrac{1}{2}  \sum^{k}_{i=0} t_{i}^{2} \norm{g_{i}}^{2}.
\]
Divide both sides of the inequality above  by $ \sum^{k}_{j=0} t_{j}$ to see that 
\[ 
 \sum^{k}_{i =0} \frac{t_{i}}{ \sum^{k}_{j=0} t_{j}} f(x^{i}, y^{i}) -   \sum^{k}_{i =0} 
 \frac{t_{i}}{ \sum^{k}_{j=0} t_{j}} f(x, y^{i}) \leq \frac{ \norm{x^{0} -x}^{2} 
 - \norm{x^{k+1} -x}^{2} }{2 \sum^{k}_{j=0} t_{j}}  
 + \frac{ \sum^{k}_{i =0}  t_{i}^{2} \norm{g_{i}}^{2} }{2 \sum^{k}_{j=0} t_{j} }, 
\]
which, combining with \cref{eq:fconcaveyhatyk}, derives \cref{eq:ineqhatxkhatyk:x}.

Similarly, sum \cref{eq:ineqtk:y} over $i \in \{0,1, \ldots, k\}$ to get that
\[ 
\sum^{k}_{i=0} t_{i} (f(x^{i}, y) - f(x^{i},y^{i})) 
\leq \tfrac{1}{2} (\norm{y^{0} -y}^{2} - \norm{y^{k+1} -y}^{2} ) 
+ \tfrac{1}{2} \sum^{k}_{i=0}  t_{i}^{2} \norm{h_{i}}^{2}.
\]
Divide also both sides of the inequality above  by $ \sum^{k}_{j=0} t_{j}$ to get  that 
\[ 
\sum^{k}_{i =0} \frac{t_{i}}{ \sum^{k}_{j=0} t_{j}} f(x^{i}, y) -   
\sum^{k}_{i =0} \frac{t_{i}}{ \sum^{k}_{j=0} t_{j}} f(x^{i},y^{i}) 
\leq \frac{ \norm{y^{0} -y}^{2} - \norm{y^{k+1} -y}^{2} }{2 \sum^{k}_{j=0} t_{j}}  
+ \frac{ \sum^{k}_{i =0}  t_{i}^{2} \norm{h_{i}}^{2} }{2 \sum^{k}_{j=0} t_{j} }, 
\]
which, connecting with \cref{eq:fconvexxhatxk}, establishes  \cref{eq:ineqhatxkhatyk:y}.
\end{proof}

 \begin{lemma} \label{lemma:geqleqineq}
Let $(x^{*}, y^{*}) \in X \times Y$ be a solution of \cref{eq:problem},
and let $k \in \mathbf{N}$. 
The following results hold. 
\begin{enumerate}
\item  \label{lemma:geqleqineq:hatxkyk} 
$ -\frac{ \norm{x^{0} -\hat{x}_{k}}^{2} - \norm{x^{k+1} -\hat{x}_{k}}^{2} }{2 \sum^{k}_{j=0} t_{j}}
  - \frac{ \sum^{k}_{i =0}  t_{i}^{2} \norm{g_{i}}^{2} }{2 \sum^{k}_{j=0} t_{j} }
  \leq f( \hat{x}_{k} ,\hat{y}_{k}) -   
 \sum^{k}_{i =0} \frac{t_{i}}{ \sum^{k}_{j=0} t_{j}} f(x^{i}, y^{i}) 
 \leq  \frac{ \norm{y^{0} -\hat{y}_{k} }^{2} 
 - \norm{y^{k+1} - \hat{y}_{k} }^{2} }{2 \sum^{k}_{j=0} t_{j}}  
 + \frac{ \sum^{k}_{i =0}  t_{i}^{2} \norm{h_{i}}^{2} }{2 \sum^{k}_{j=0} t_{j} }$.
Consequently,
\[ 
-\frac{ \norm{x^{0} -\hat{x}_{k}}^{2}   }{2 \sum^{k}_{j=0} t_{j}}  
- \frac{ \sum^{k}_{i =0}  t_{i}^{2} \norm{g_{i}}^{2} }{2 \sum^{k}_{j=0} t_{j} } 
\leq f( \hat{x}_{k} ,\hat{y}_{k}) -   
\sum^{k}_{i =0} \frac{t_{i}}{ \sum^{k}_{j=0} t_{j}} f(x^{i}, y^{i}) 
\leq  \frac{ \norm{y^{0} -\hat{y}_{k} }^{2}  }{2 \sum^{k}_{j=0} t_{j}}    
+ \frac{ \sum^{k}_{i =0}  t_{i}^{2} \norm{h_{i}}^{2} }{2 \sum^{k}_{j=0} t_{j} }.
\]
\item  \label{lemma:geqleqineq:x*y*} 
$- \frac{ \norm{y^{0} -y^{*}}^{2} - \norm{y^{k+1} -y^{*}}^{2} }{2 \sum^{k}_{j=0} t_{j}}   
 - \frac{ \sum^{k}_{i =0}  t_{i}^{2} \norm{h_{i}}^{2} }{2 \sum^{k}_{j=0} t_{j} } 
 \leq \sum^{k}_{i =0} \frac{t_{i}}{ \sum^{k}_{j=0} t_{j}} f(x^{i}, y^{i}) -  f(x^{*}, y^{*}) 
 \leq \frac{ \norm{x^{0} -x^{*}}^{2} - \norm{x^{k+1} -x^{*}}^{2} }{2 \sum^{k}_{j=0} t_{j}}  
 + \frac{ \sum^{k}_{i =0}  t_{i}^{2} \norm{g_{i}}^{2} }{2 \sum^{k}_{j=0} t_{j} }  $.
Consequently, 
\[  
- \frac{ \norm{y^{0} -y^{*}}^{2}  }{2 \sum^{k}_{j=0} t_{j}}    
- \frac{ \sum^{k}_{i =0}  t_{i}^{2} \norm{h_{i}}^{2} }{2 \sum^{k}_{j=0} t_{j} } 
\leq \sum^{k}_{i =0} \frac{t_{i}}{ \sum^{k}_{j=0} t_{j}} f(x^{i}, y^{i}) 
-  f(x^{*}, y^{*}) \leq \frac{ \norm{x^{0} -x^{*}}^{2} }{2 \sum^{k}_{j=0} t_{j}}  
+ \frac{ \sum^{k}_{i =0}  t_{i}^{2} \norm{g_{i}}^{2} }{2 \sum^{k}_{j=0} t_{j} }.  
\]
 \item  \label{lemma:geqleqineq:sumabove} 
 $ -\frac{ \norm{x^{0} -\hat{x}_{k}}^{2} + \norm{y^{0} -y^{*}}^{2} }{2 \sum^{k}_{j=0} t_{j}} 
  - \frac{ \sum^{k}_{i =0}  t_{i}^{2} (\norm{g_{i}}^{2} 
  + \norm{h_{i}}^{2}  )}{2 \sum^{k}_{j=0} t_{j} } \leq f( \hat{x}_{k} ,\hat{y}_{k}) 
  -  f(x^{*}, y^{*})  \leq  \frac{ \norm{x^{0} -x^{*}}^{2} 
  + \norm{y^{0} -\hat{y}_{k} }^{2}   }{2 \sum^{k}_{j=0} t_{j}}    
  + \frac{ \sum^{k}_{i =0}  t_{i}^{2}(  \norm{g_{i}}^{2} 
  + \norm{h_{i}}^{2}) }{2 \sum^{k}_{j=0} t_{j} }$.
 	\end{enumerate} 
 \end{lemma}
 
 \begin{proof}
 	\cref{lemma:geqleqineq:hatxkyk}: 
	Substitute $x$ in \cref{eq:ineqhatxkhatyk:x}  of \cref{lemma:xkyk} 
	with  $\hat{x}_{k}$ to observe that 
\[ 
  \sum^{k}_{i =0} \frac{t_{i}}{ \sum^{k}_{j=0} t_{j}} f(x^{i}, y^{i}) 
  -  f(\hat{x}_{k},\hat{y}_{k} ) \leq \frac{ \norm{x^{0} -\hat{x}_{k}}^{2} 
  - \norm{x^{k+1} -\hat{x}_{k}}^{2} }{2 \sum^{k}_{j=0} t_{j}}  
  + \frac{ \sum^{k}_{i =0}  t_{i}^{2} \norm{g_{i}}^{2} }{2 \sum^{k}_{j=0} t_{j} }. 
\]
Replace $y$ in \cref{eq:ineqhatxkhatyk:y} of \cref{lemma:xkyk} by $\hat{y}_{k}$ to get that 
\[ 	
 f( \hat{x}_{k} ,\hat{y}_{k}) -   \sum^{k}_{i =0} \frac{t_{i}}{ \sum^{k}_{j=0} t_{j}} f(x^{i}, y^{i}) 
 \leq  \frac{ \norm{y^{0} -\hat{y}_{k}}^{2} - 
 \norm{y^{k+1} -\hat{y}_{k}}^{2} }{2 \sum^{k}_{j=0} t_{j}}    
 + \frac{ \sum^{k}_{i =0}  t_{i}^{2} \norm{h_{i}}^{2} }{2 \sum^{k}_{j=0} t_{j} }.
\]
 Combine the last two inequalities to deduce the desired inequality 
 in \cref{lemma:geqleqineq:hatxkyk}. 
 
\cref{lemma:geqleqineq:x*y*}: 
Note that $(x^{k})_{k \in \mathbf{N}}$ and  $(y^{k})_{k \in \mathbf{N}}$ 
are   sequences in $X$ and $Y$, respectively,  and   
that $X$ and $Y$ are convex. 
We know that   $ \hat{x}_{k} \in X $  and $\hat{y}_{k} \in Y$. 
Recall that $(x^{*}, y^{*}) \in X \times Y$ is a solution of \cref{eq:problem}. 
Employing \cref{eq:solution}, we obtain that 
\begin{subequations}
\begin{align}
& f(x^{*}, \hat{y}_{k})  \leq  f(x^{*}, y^{*}); \label{eq:x*hatyk}\\
& f(x^{*}, y^{*})   \leq f(\hat{x}_{k} , y^{*}). \label{eq:hatxky*}
\end{align}
\end{subequations}

Set $x = x^{*}$ in \cref{eq:ineqhatxkhatyk:x} to get that 
\[ 
\sum^{k}_{i =0} \frac{t_{i}}{ \sum^{k}_{j=0} t_{j}} f(x^{i}, y^{i}) 
-  f(x^{*},\hat{y}_{k} ) \stackrel{\cref{eq:ineqhatxkhatyk:x}}{  \leq} \frac{ \norm{x^{0} -x^{*}}^{2} 
- \norm{x^{k+1} -x^{*}}^{2} }{2 \sum^{k}_{j=0} t_{j}}  
+ \frac{ \sum^{k}_{i =0}  t_{i}^{2} \norm{g_{i}}^{2} }{2 \sum^{k}_{j=0} t_{j} },  
\]
 which, connecting with  \cref{eq:x*hatyk},  derives 
\begin{align} \label{eq:lemma:geqleqineq:x*}
\sum^{k}_{i =0} \frac{t_{i}}{ \sum^{k}_{j=0} t_{j}} f(x^{i}, y^{i}) 
-   f(x^{*}, y^{*})   \leq  \frac{ \norm{x^{0} -x^{*}}^{2} 
- \norm{x^{k+1} -x^{*}}^{2} }{2 \sum^{k}_{j=0} t_{j}}  
+ \frac{ \sum^{k}_{i =0}  t_{i}^{2} \norm{g_{i}}^{2} }{2 \sum^{k}_{j=0} t_{j} },  
\end{align}
 
 Similarly, setting $y = y^{*}$ in \cref{eq:ineqhatxkhatyk:y}, we have that 
\[ 
f( \hat{x}_{k} ,y^{*}) -   \sum^{k}_{i =0} \frac{t_{i}}{ \sum^{k}_{j=0} t_{j}} f(x^{i}, y^{i}) \stackrel{\cref{eq:ineqhatxkhatyk:y} }{\leq}   
\frac{ \norm{y^{0} -y^{*}}^{2} - \norm{y^{k+1} -y^{*}}^{2} }{2 \sum^{k}_{j=0} t_{j}}    
+ \frac{ \sum^{k}_{i =0}  t_{i}^{2} \norm{h_{i}}^{2} }{2 \sum^{k}_{j=0} t_{j} },
\]
which, combining with \cref{eq:hatxky*}, entails
 \begin{align}\label{eq:lemma:geqleqineq:y*}
f( x^{*} ,y^{*}) -   \sum^{k}_{i =0} \frac{t_{i}}{ \sum^{k}_{j=0} t_{j}} f(x^{i}, y^{i}) 
\leq  \frac{ \norm{y^{0} -y^{*}}^{2} - \norm{y^{k+1} -y^{*}}^{2} }{2 \sum^{k}_{j=0} t_{j}}    
+ \frac{ \sum^{k}_{i =0}  t_{i}^{2} \norm{h_{i}}^{2} }{2 \sum^{k}_{j=0} t_{j} }.
\end{align}
 
 Combine \cref{eq:lemma:geqleqineq:x*} and \cref{eq:lemma:geqleqineq:y*} 
 to reach the required inequality in \cref{lemma:geqleqineq:x*y*}.
 
\cref{lemma:geqleqineq:sumabove}: This follows immediately from 
\cref{lemma:geqleqineq:hatxkyk}  and \cref{lemma:geqleqineq:x*y*} above.
 \end{proof}
 
 The following result is inspired by \cite[Proposition~3.1]{NedicOzdaglar2009} 
 in which, under some extra  boundedness or compactness assumptions,  
 the authors considered the convergence 
 $\frac{1}{ k+1}	\sum^{k}_{i =0}  f(x^{i}, y^{i}) \to f(x^{*}, y^{*}) $  
 and $f(\tfrac{1}{k}\sum^{k-1}_{i=0} x^{i}, \tfrac{1}{k}\sum^{k-1}_{i=0} y^{i}) \to f(x^{*},y^{*}) $, 
 where $(x^{*},y^{*}) $ is a saddle-point of $f$ and  
 $\left((x^{k}, y^{k})\right)_{k \in \mathbf{N}}$   is generated 
 by \cref{eq:algorithm} with $(\forall k \in \mathbf{N})$ $t_{k} \equiv \alpha \in 
 \mathbf{R}_{++}$. 
  
  \begin{theorem} \label{theorem:sumconverge}
  Let $(x^{*}, y^{*}) \in X \times Y$ be a solution of \cref{eq:problem}.		
  Let $R$, $G$, and  $S$ be in $\mathbf{R}_{++}$. Suppose that 
\begin{align} \label{prop:sumconverge:RG}
  \norm{(x^{0},y^{0})} \leq R, \quad \norm{(x^{*},y^{*})} \leq R, 
   \quad \text{and}  \quad (\forall k \in \mathbf{N}) \norm{(g_{k}, h_{k})} \leq G, 
\end{align}
 and that the step sizes satisfy that
\begin{align} \label{prop:sumconverge:tiS}
	(\forall i \in \mathbf{N})~ t_{i} \geq 0 \text{ with } t_{0} > 0, \quad  
	\sum^{\infty}_{j=0} t_{j} =\infty, \quad \text{and} \quad   \sum^{\infty}_{j=0} 
	t_{j}^{2} 
	=S < \infty. 
\end{align}

Then  $\sum^{k}_{i =0} \frac{t_{i}}{ \sum^{k}_{j=0} t_{j}} f(x^{i}, y^{i}) $  
converges to  $ f(x^{*}, y^{*}) $.
  \end{theorem}

\begin{proof}
According to \cref{lemma:geqleqineq}\cref{lemma:geqleqineq:x*y*}, we have that 
\[ 
- \frac{ \norm{y^{0} -y^{*}}^{2}  }{2 \sum^{k}_{j=0} t_{j}}    
- \frac{ \sum^{k}_{i =0}  t_{i}^{2} \norm{h_{i}}^{2} }{2 \sum^{k}_{j=0} t_{j} } 
\leq \sum^{k}_{i =0} \frac{t_{i}}{ \sum^{k}_{j=0} t_{j}} f(x^{i}, y^{i}) -  f(x^{*}, y^{*}) \leq \frac{ \norm{x^{0} -x^{*}}^{2} }{2 \sum^{k}_{j=0} t_{j}}  + \frac{ \sum^{k}_{i =0}  t_{i}^{2} \norm{g_{i}}^{2} }{2 \sum^{k}_{j=0} t_{j} },
\]
which, combined with our assumptions,
\cref{prop:sumconverge:RG} and \cref{prop:sumconverge:tiS}, 
guarantees that
\[ 
- \frac{ 4R^{2}  }{2 \sum^{k}_{j=0} t_{j}}    
- \frac{ G^{2}S}{2 \sum^{k}_{j=0} t_{j} } 
\leq \sum^{k}_{i =0} \frac{t_{i}}{ \sum^{k}_{j=0} t_{j}} f(x^{i}, y^{i}) 
-  f(x^{*}, y^{*}) \leq \frac{ 4R^{2} }{2 \sum^{k}_{j=0} t_{j}}  
+ \frac{ G^{2}S}{2 \sum^{k}_{j=0} t_{j} },
\]
The inequalities above ensure  
$ \sum^{k}_{i =0} \frac{t_{i}}{ \sum^{k}_{j=0} t_{j}} f(x^{i}, y^{i}) -  f(x^{*}, y^{*}) \to 0$ 
since  $ \sum^{\infty}_{i=0} t_{i} =\infty$.
\end{proof}

Below we shall prove the boundedness of the sequence 
$((x^{k},y^{k}))_{k \in \mathbf{N}}$, 
which plays a critical role in our proof of 
\cref{theorem:xhatkyhatkConverge} below.

\begin{lemma} \label{lemma:basicsubgradients}
Let $(x^{*}, y^{*}) \in X \times Y$ be a solution of \cref{eq:problem}.		
Then 
\[  
	(\forall k \in \mathbf{N}) \quad \innp{(x^{k},y^{k})-(x^{*},y^{*}), (g_{k},h_{k}) } 
	\geq 
	0. 
\] 
\end{lemma}	

\begin{proof}
As a result of 	\cite[Theorem~1]{Rockafellar1970}, the operator
\[ 
T : X \times Y \to \mathcal{H}_{1}  \times \mathcal{H}_{2}: (\bar{x},\bar{y}) \mapsto (u,v), 
\quad \text{where } u \in \partial_{x}f (\bar{x},\bar{y}) 
\text{ and } v \in  \partial_{y}( -f (\bar{x},\bar{y}))
\]
is monotone. 

In view of our assumption and \cref{fact:saddlepoint}, 
\begin{align} \label{eq:lemma:basicsubgradients}
0 \in \partial_{x} f(x^{*},y^{*})  \quad and \quad 0 \in \partial_{y}(-f(x^{*},y^{*})).
\end{align}
Recall from the construction of the algorithm \cref{eq:algorithm}  
that $ (\forall k \in \mathbf{N})$  $g_{k} \in \partial_{x}f(x^{k},y^{k}) $ 
and $h_{k} \in \partial_{y}(- f(x^{k},y^{k}) )$.
Combine this with \cref{eq:lemma:basicsubgradients}  to derive that
\[ 
(\forall k \in \mathbf{N}) \quad  
\innp{(x^{k},y^{k})-(x^{*},y^{*}), (g_{k},h_{k}) } 
=   \innp{(x^{k},y^{k})-(x^{*},y^{*}), T(x^{k},y^{k})-T(x^{*},y^{*})  } \geq 0,
\]
where we use the monotonicity of $T$ in the last inequality.  
\end{proof}

\begin{lemma} \label{lemma:xkykbounded}
Let $(x^{*}, y^{*}) \in X \times Y$ be a solution of \cref{eq:problem}.	 
Let $R$, $G$, and  $S$ be in $\mathbf{R}_{++}$. Suppose that 
\[ 
\norm{(x^{0},y^{0})} \leq R, \quad \norm{(x^{*},y^{*})} \leq R,  
\quad \text{and}  \quad  (\forall k \in \mathbf{N})  \norm{(g_{k}, h_{k})} \leq G, 
\]
and that the step sizes satisfy that
\[ 
(\forall i \in \mathbf{N}) ~ t_{i} \geq 0 \text{ with } t_{0} > 0   \quad \text{and} 
\quad   
\sum^{\infty}_{j=0} t_{j}^{2} =S < \infty. 
\]
Then $((x^{k}, y^{k}))_{k \in \mathbf{N}}$ is bounded.
\end{lemma}

\begin{proof}
Let $k \in \mathbf{N}$.	
Applying \cref{eq:algorithm} in the first equation, we have that 
\begin{align*}
&\norm{ (x^{k+1}, y^{k+1})  - (x^{*}, y^{*})}^{2}\\
= & \norm{ (\Pro_{X}(x^{k} - t_{k} g_{k}), \Pro_{Y}(y^{k} - t_{k} h_{k}))  
- (x^{*}, y^{*}) }^{2}\\
= & \norm{\Pro_{X}(x^{k} - t_{k} g_{k}) - x^{*}}^{2} 
+ \norm{\Pro_{Y}(y^{k} - t_{k} h_{k}) -  y^{*}}^{2}\\
= &\norm{\Pro_{X}(x^{k} - t_{k} g_{k}) - \Pro_{X}x^{*}}^{2} 
+ \norm{\Pro_{Y}(y^{k} - t_{k} h_{k}) -  \Pro_{Y}y^{*}}^{2}\\
\leq & \norm{ x^{k} - t_{k} g_{k} -x^{*} }^{2} 
+  \norm{ y^{k} - t_{k} h_{k} - y^{*}}^{2}\\
=  & \norm{x^{k} -x^{*} }^{2} - 2 t_{k} \innp{x^{k}-x^{*}, g_{k}} 
+ t_{k}^{2}\norm{g_{k}}^{2} + \norm{y^{k} -y^{*} }^{2} 
-2 t_{k} \innp{y^{k}-y^{*}, h_{k}} + t_{k}^{2} \norm{h_{k}}^{2}\\
= &\norm{(x^{k},y^{k})- (x^{*},y^{*}) }^{2} 
-2t_{k} \innp{(x^{k},y^{k})-(x^{*},y^{*}), (g_{k},h_{k}) } + t_{k}^{2} \norm{(g_{k},h_{k})}^{2}.
\end{align*}
Note that we use \cref{eq:norm} in the second equation,  
that we use the fact $x^{*} \in X$, $y^{*} \in Y$, 
$x^{*} =\Pro_{X} x^{*}$, and $y^{*} =\Pro_{Y} y^{*}$ in the third equation, 
that the nonexpansiveness of $\Pro_{X}$ and $\Pro_{Y}$ 
is used in the inequality above, and that we use both \cref{eq:norm} 
and  \cref{eq:innerproduct} in the last equation. 
	
The result above actually tells us that for every $ i \in \{0,1,\ldots,k\}$,
\[ 
\norm{ (x^{i+1}, y^{i+1})  - (x^{*}, y^{*})}^{2} -\norm{(x^{i},y^{i})- (x^{*},y^{*}) }^{2} 
+ 2t_{i} \innp{(x^{i},y^{i})-(x^{*},y^{*}), (g_{i},h_{i}) }  
\leq t_{i}^{2} \norm{(g_{i},h_{i})}^{2}.
\]
Sum the inequality above over   $ i \in \{0,1,\ldots,k\}$ to obtain that
\begin{align*}
	&\norm{ (x^{k+1}, y^{k+1})  - (x^{*}, y^{*})}^{2} -\norm{(x^{0},y^{0})- 
	(x^{*},y^{*}) 
	}^{2} + 2 \sum^{k}_{i=0} t_{i} \innp{(x^{i},y^{i})-(x^{*},y^{*}), (g_{i},h_{i}) }\\
	\leq &   \sum^{k}_{i=0}t_{i}^{2} \norm{(g_{i},h_{i})}^{2},
\end{align*}
which, combined with the assumptions,  implies that 
\begin{align}\label{eq:lemma:xkykbounded:leq}
\norm{ (x^{k+1}, y^{k+1})  - (x^{*}, y^{*})}^{2}
+ 2 \sum^{k}_{i=0} t_{i} \innp{(x^{i},y^{i})-(x^{*},y^{*}), (g_{i},h_{i}) } 
 \leq 4R^{2} + SG^{2}.
\end{align}

 In view of \cref{lemma:basicsubgradients}, 
 $\sum^{k}_{i=0} t_{i} \innp{(x^{i},y^{i})-(x^{*},y^{*}), (g_{i},h_{i}) } \geq 0$. 
 Hence,  by \cref{eq:lemma:xkykbounded:leq},
\[ 
\norm{ (x^{k+1}, y^{k+1})  - (x^{*}, y^{*})}^{2}   \leq 4R^{2} + SG^{2},
\]
which yields the
boundedness of $((x^{k}, y^{k}))_{k \in \mathbf{N}}$.
\end{proof}

\begin{lemma} \label{lemma:xhatkyhatkbounded}
Let $(x^{*}, y^{*}) \in X \times Y$ be a solution of \cref{eq:problem}.		
Let $R$, $G$, and  $S$ be in $\mathbf{R}_{++}$. Suppose that 
\[ 
\norm{(x^{0},y^{0})} \leq R, \quad \norm{(x^{*},y^{*})} \leq R,  
\quad \text{and}  \quad (\forall k \in \mathbf{N})  \norm{(g_{k}, h_{k})} \leq G, 
\]  
and that the step sizes satisfy that
\[  
(\forall i \in \mathbf{N}) ~t_{i} \geq 0 \text{ with } t_{0} > 0   \quad \text{and} 
\quad   \sum^{\infty}_{j=0} t_{j}^{2} =S < \infty. 
\]

Then $((\hat{x}_{k}, \hat{y}_{k}))_{k \in \mathbf{N}}$ is bounded.
\end{lemma}

\begin{proof}
The required result is clear from the definition of the sequence $((\hat{x}_{k}, 
\hat{y}_{k}))_{k \in \mathbf{N}}$ and the boundedness of $((x^{k}, y^{k}))_{k 
\in \mathbf{N}}$  proved in \cref{lemma:xkykbounded}.
\end{proof}

 \begin{theorem} \label{theorem:xhatkyhatkConverge}
 Let $(x^{*}, y^{*}) \in X \times Y$ be a solution of \cref{eq:problem}.		
 Let $R$, $G$, and  $S$ be in $\mathbf{R}_{++}$. Suppose that 
\[   
\norm{(x^{0},y^{0})} \leq R, \quad \norm{(x^{*},y^{*})} \leq R,  \quad \text{and}  
\quad (\forall k \in \mathbf{N})  \norm{(g_{k}, h_{k})} \leq G, 
\]
 and that the step sizes satisfy that
\[ 
(\forall i \in \mathbf{N}) ~ t_{i} \geq 0 \text{ with } t_{0} > 0, \quad  
\sum^{\infty}_{j=0} t_{j} =\infty, \quad \text{and} \quad   \sum^{\infty}_{j=0} 
t_{j}^{2} =S < \infty. 
\]
 
Then   $f( \hat{x}_{k} ,\hat{y}_{k}) $ converges to $  f(x^{*}, y^{*})$.
 \end{theorem}

\begin{proof}
As a result of our assumptions and \cref{lemma:xhatkyhatkbounded}, 
there exists a constant $Q \in \mathbf{R}_{++}$ such that 
\[ 
(\forall k \in \mathbf{N}) \quad \norm{ ( \hat{x}_{k} ,\hat{y}_{k}) } \leq Q.
\]
Combine this with our assumptions and 
\cref{lemma:geqleqineq}\cref{lemma:geqleqineq:sumabove} to establish that
\[ 
-\frac{ 4R^{2} + (R+Q)^{2}  }{2 \sum^{k}_{j=0} t_{j}}  
- \frac{ SG^{2} }{2 \sum^{k}_{j=0} t_{j} } \leq f( \hat{x}_{k} ,\hat{y}_{k}) -  f(x^{*}, y^{*})  
\leq  \frac{ 4R^{2} + (R+Q)^{2}  }{2 \sum^{k}_{j=0} t_{j}}    
+ \frac{ SG^{2} }{2 \sum^{k}_{j=0} t_{j} },
\]
which, combining with the assumption $\sum^{\infty}_{j=0} t_{j} =\infty$, 
entails that  $f( \hat{x}_{k} ,\hat{y}_{k}) \to  f(x^{*}, y^{*})  $.
\end{proof}

\section{Numerical experiments} \label{section:NumericalExperiments}
In this section, we implement our alternating subgradient method on some 
particular examples to verify our convergence results presented in previous
\cref{section:ASM} and also to analyze convergence rates of our algorithms. 
Moreover, we compare our iterate scheme with step sizes
$(t_{k})_{k \in \mathbf{N}}$ and iterate schemes 
considered in \cite{NedicOzdaglar2009}
with constant  step sizes. 
Based on our numerical results presented
in  \cref{fig:subgradient_methods_toy_0.png},
we see benefits of replacing constant step sizes by step sizes 
$(t_{k})_{k \in \mathbf{N}}$  satisfying
$(\forall i \in \mathbf{N}\smallsetminus \{0\} )$ $t_{i} \geq 0 $, $t_{0} > 0$,
$\sum^{\infty}_{j=0} t_{j} =\infty$, and $\sum^{\infty}_{j=0} t_{j}^{2} < \infty$. 
In addition, in view of \cref{fig:subgradient_methods_toy_t_0.png},
\cref{fig:subgradient_methods_lp_random_0.png},
\cref{fig:subgradient_methods_lsl1_random_0.png}, 
and \cref{fig:app_rmpc_0.png} below obtained from our numerical experiments,
we discover the convergence $f(x^{k},y^{k}) \to f(x^{*},y^{*})$
in multiple examples,
which doesn't have any theoretical support yet.

We know that generally subgradient methods converge slowly. 
Normally, the iterate point corresponding to larger iterate number 
has better convergence performance. 
Therefore, to accelerate our subgradient methods for solving 
convex-concave saddle-point problems, in our experiments below 
we reorder sequence $(t_{k})_{k \in \mathbf{N}}$ 
in descending order.
Based on our numerical results, this trick works very well.

In this section, unless stated otherwise, 
$\left((x^{k}, y^{k})\right)_{0 \leq k \leq K}$    
is generated by \cref{eq:algorithm} with $K$ being the number of iterates
and step sizes $(\forall k \in \{1, \ldots, K\})$ $t_{k} =\frac{1}{K+1 -k}$.
Moreover, we have
\[ 
(\forall k \in \mathbf{N})\quad (\hat{x}_{k},\hat{y}_{k})=\left(\sum^{k}_{i =0} 
\frac{t_{i}}{ \sum^{k}_{j=0} t_{j}} x^{i}, \sum^{k}_{i =0} \frac{t_{i}}{ \sum^{k}_{j=0} 
	t_{j}} y^{i}\right).
\] 
Note that the choice of step sizes
\[ (\forall k \in \{1, \ldots, K\}) \quad t_{k} =\frac{1}{K+1 -k} \]
satisfies  
$(\forall i \in \mathbf{N}\smallsetminus \{0\} )$ $t_{i} \geq 0 $, $t_{0} > 0$,
$\sum^{\infty}_{j=0} t_{j} =\infty$, and $\sum^{\infty}_{j=0} t_{j}^{2} < \infty$,
which is required in our convergence results 
$\sum^{k}_{i =0} \frac{t_{i}}{ \sum^{k}_{j=0} t_{j}} f(x^{i}, y^{i}) \to  f(x^{*}, y^{*}) $
and 
$f( \hat{x}_{k} ,\hat{y}_{k}) \to  f(x^{*}, y^{*})$
provided in \cref{theorem:sumconverge} and  \cref{theorem:xhatkyhatkConverge},
respectively.

\subsection{Toy example} \label{subsection:toyexample}
In \cite[Proposition~3.1]{NedicOzdaglar2009},   
the authors considered the convergence of 
$\frac{1}{ k+1}	\sum^{k}_{i =0}  f(x^{i}, y^{i}) \to f(x^{*}, y^{*}) $  
and 
$f(\tfrac{1}{k}\sum^{k-1}_{i=0} x^{i}, \tfrac{1}{k}\sum^{k-1}_{i=0} y^{i}) \to 
f(x^{*},y^{*}) $ 
within certain error level and with some boundedness and compactness assumptions, 
where $f:X\times Y \to \mathbf{R}$ is a convex-concave function,  $(x^{*},y^{*}) $ is a 
saddle-point of $f$, and  $\left((x^{k}, y^{k})\right)_{k \in \mathbf{N}}$
is generated by \cref{eq:algorithm} 
with $(\forall k \in \mathbf{N})$ $t_{k} \equiv \alpha \in \mathbf{R}_{++}$. 

In this subsection, we compare our iterate scheme with 
iterate schemes associated with step sizes being a constant 
to show the drawback of constant step sizes.

We consider the toy convex-concave function  $f : \mathbf{R} \times \mathbf{R}_{+}  
\to \mathbf{R}$ defined by 
\[ 
(\forall (x,y) \in \mathbf{R} \times \mathbf{R}_{+}) \quad f(x,y) = x^2 (1+y) - 6xy +8y+1.
\]
which is considered in \cref{example:lagrange}\cref{example:lagrange:toy}.  As a 
consequence of 
\cref{fact:saddlepoint}, $(2,2)$ is a saddle-point of $f$. So all desired sequences 
of iterates must converge to  $f(2,2) =5$. 

In our experiments related to \cref{fig:subgradient_methods_toy_0.png}, 
we mainly consider
$\left(	\sum^{k}_{i =0}  \frac{t_{i}}{ \sum^{k}_{j =0}t_{j} }  f(x^{i}, y^{i})  \right)_{k \in 
 \mathbf{N}}$ and $\left(  f(\hat{x}_{k},\hat{y}_{k}) \right)_{k \in \mathbf{N}}$ 
with  $(\forall k \in \mathbf{N})$ $t_{i} =\alpha \in \mathbf{R}_{++}$ and
for  different values of $\alpha$. 
Note that although when $(\forall k \in \mathbf{N})$ $t_{i} =\alpha \in \mathbf{R}_{++}$,
$(\forall i \in \mathbf{N})$  $\frac{t_{i}}{ \sum^{k}_{j =0}t_{j} }  =\frac{1}{k+1}$
 independent of $\alpha$, the sequence of iterates $((x^{k},y^{k}))_{k \in \mathbf{N}}$ 
 generated by the iterate scheme \cref{eq:algorithm} is dependent on the sequence 
 $(t_{i})_{i\in \mathbb{N}}$. 
 So  different values of $\alpha$ indeed deduce different sequences of iterates in 
 consideration.
 Based on our results, when $\alpha >1$, 
 both $\left( \frac{1}{ k+1}	\sum^{k}_{i =0}  f(x^{i}, y^{i})  \right)_{k \in \mathbf{N}}$ 
 and $\left(   f(\hat{x}_{k},\hat{y}_{k}) \right)_{k \in \mathbf{N}}$  
 generally don't converge to the desired  value $5$. 
 
 To get the following 
 \cref{fig:subgradient_methods_toy_0.png}, 
we calculate $\left( \frac{1}{ k+1}	\sum^{k}_{i =0}  f(x^{i}, y^{i})  \right)_{0 \leq k \leq 
200}$ and $\left(  f(\hat{x}_{k},\hat{y}_{k}) \right)_{0 \leq k \leq 200}$ with a 
random chosen initial point and with setting 
$\alpha =1, \alpha =0.8, \alpha =0.5, \alpha = 0.1, \alpha =0.01$, 
and $\alpha =0.0001$, respectively.    
Moreover, we also calculate 
$\left( \frac{1}{ k+1}	\sum^{k}_{i =0}  f(x^{i}, y^{i})  \right)_{0 \leq k \leq 200}$ 
and $\left(  f(\hat{x}_{k},\hat{y}_{k}) \right)_{0 \leq k \leq 200}$ with
$(\forall k \in \{1, \ldots, 200\})$ $t_{k} =\frac{1}{200+1 -k}$ as a reference.

According to the first subplot of \cref{fig:subgradient_methods_toy_0.png}, 
we observe that  
generally sequences of iterates associated with constant step sizes
don't converge to required optimal values.
(In our numerous related experiments for this particular example,  
we found that only when $(\forall i \in \mathbf{N})$ $t_{i}=0.1$,  the sequences 
$\left( \frac{1}{ k+1}	\sum^{k}_{i =0}  f(x^{i}, y^{i})  \right)_{k \in \mathbf{N}}$ 
and $\left(  f(\hat{x}_{k},\hat{y}_{k}) \right)_{k \in  \mathbf{N}}$ 
converge to the required value $5$ consistently, 
regardless of the random initial points and problem data. )
Moreover, the second subplot of \cref{fig:subgradient_methods_toy_0.png}
shows that the convergence rate of  
our iterate scheme with $(\forall k \in \{1, \ldots, 200\})$ $ t_{k} =\frac{1}{200+1 -k} $
is faster than iterates schemes associated with constant step sizes.
\cref{fig:subgradient_methods_toy_0.png} shows some drawbacks of constant step 
sizes
and explains why should we consider $(t_{k})_{k \in \mathbf{N}}$ not being a sequence 
of constant. 

In fact, in our experiments associated with this example,
we calculated $(  f(x^{k},y^{k}) )_{1 \leq k \leq 200}$ together with 
$\left( \frac{1}{ k+1}	\sum^{k}_{i =0}  f(x^{i}, y^{i})  \right)_{1 \leq k \leq 200}$ and 
$\left(  f(\hat{x}_{k},\hat{y}_{k}) \right)_{1 \leq k \leq 200}$.
We don't present them in \cref{fig:subgradient_methods_toy_0.png} because 
when step sizes are constant,
the performance of   $(  f(x^{k},y^{k}) )_{1 \leq k \leq 200}$ is much worse than 
performances of sequences presented in \cref{fig:subgradient_methods_toy_0.png} 
below.

\begin{figure}[H]
\centering
\includegraphics[width=0.8\textwidth]{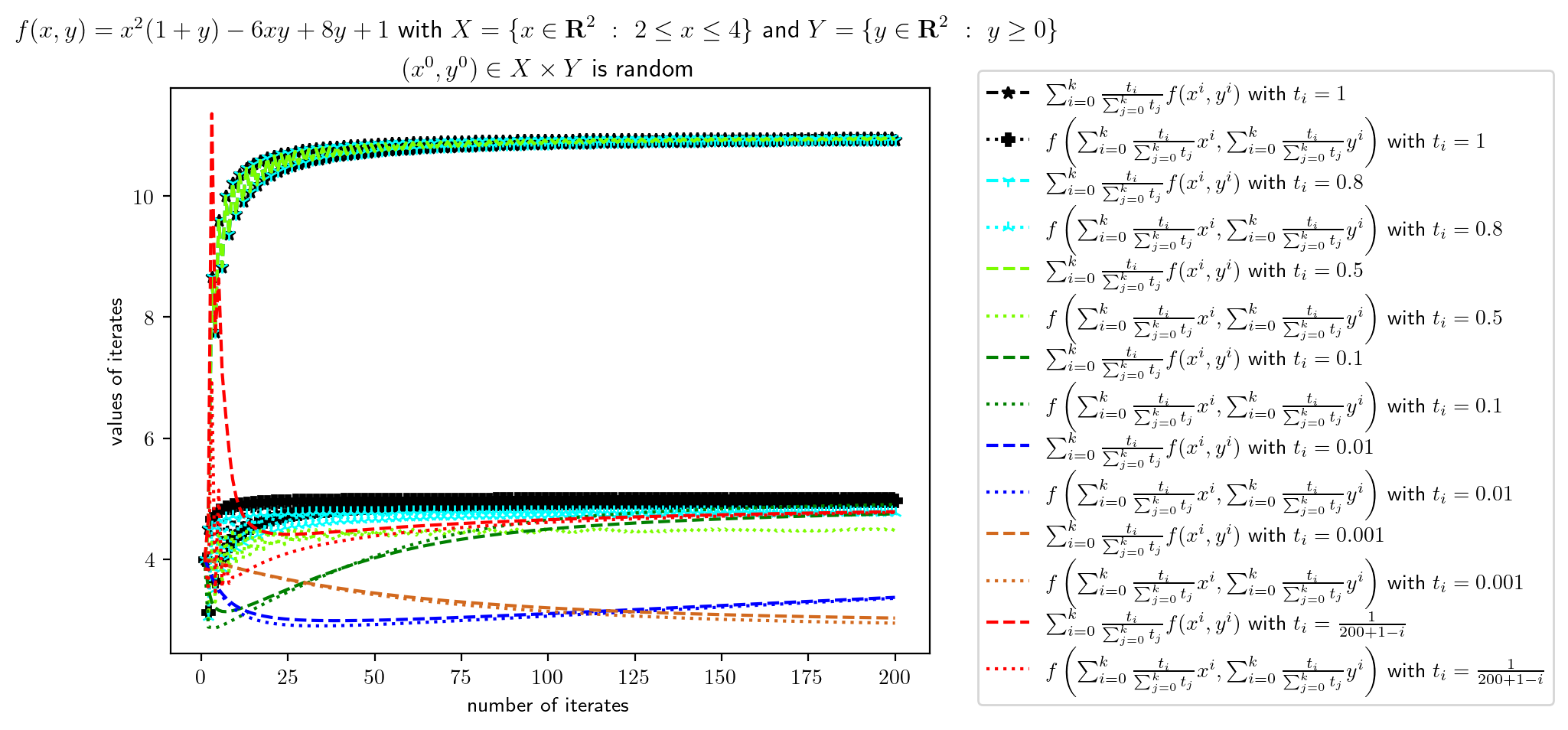}
\includegraphics[width=0.8\textwidth]{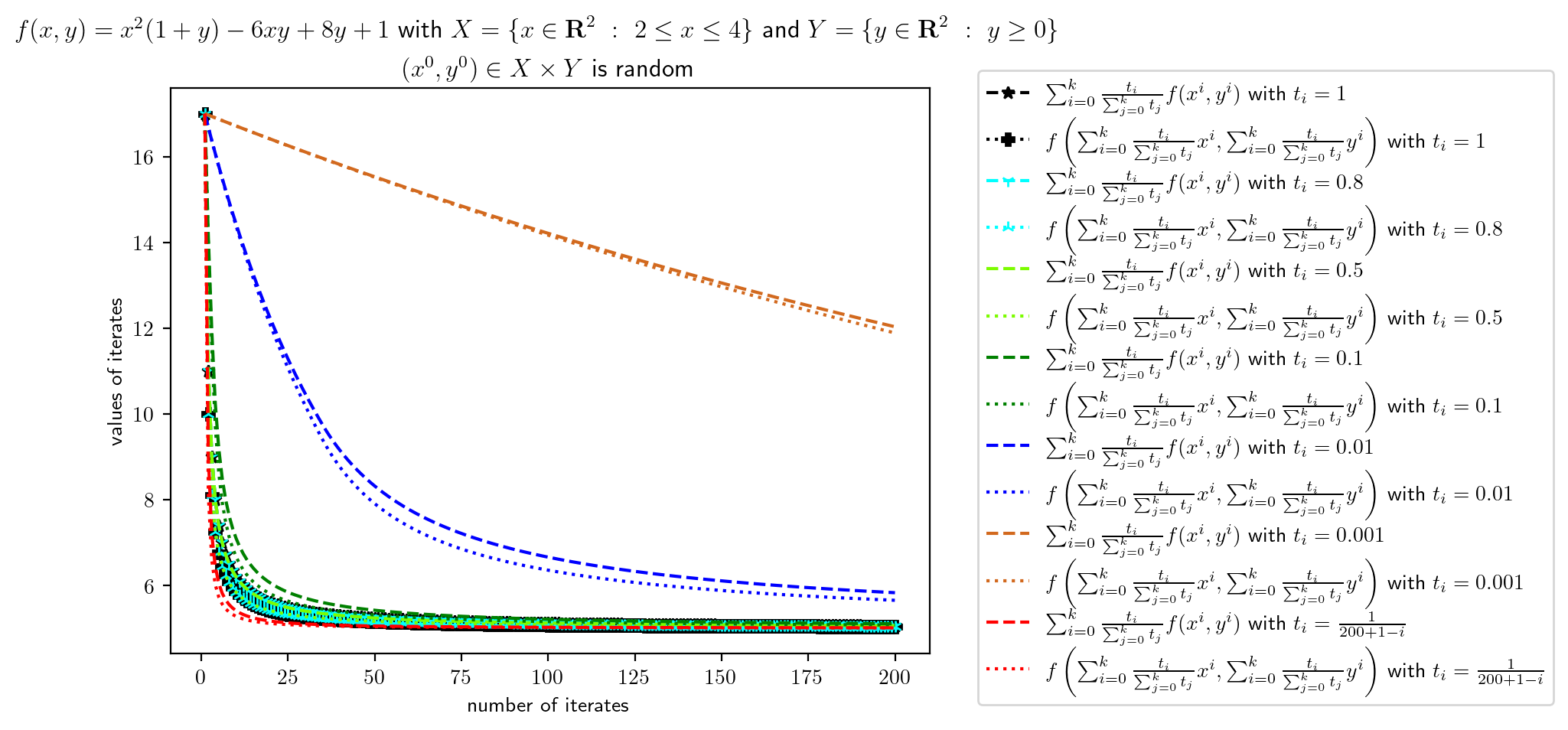}
\caption{Comparison with iterate scheme associated with constant step sizes}
\label{fig:subgradient_methods_toy_0.png}
\end{figure}

Then  we randomly chose  initial points and calculate  the sequence $\left(	
\sum^{k}_{i =0}  \frac{t_{i}}{ 
\sum^{k}_{j =0}t_{j} }  f(x^{i}, y^{i})  
\right)_{1 \leq k \leq 500}$, $\left(  f(\hat{x}_{k},\hat{y}_{k}) \right)_{1 \leq k \leq 
500}$, and $(  f(x^{k},y^{k}) )_{1 \leq k \leq 500}$ with  
$(\forall k \in \{1, \ldots, 500\})$ $t_{k} =\frac{1}{500+1 -k}$.
 Our result is presented in \cref{fig:subgradient_methods_toy_t_0.png} below. 
 Note that to get a clearer view on the convergence rate, we zoom 
 in more interesting part  and set the range of y-axis view as   $[4.5, 5.5]$.
 
 The theoretical convergence of $\sum^{k}_{i =0}  \frac{t_{i}}{ 
 	\sum^{k}_{j =0}t_{j} }  f(x^{i}, y^{i})   \to f(x^{*},y^{*})$  and 
 $f(\hat{x}_{k},\hat{y}_{k})  \to f(x^{*},y^{*}) $ is 
 presented in \cref{theorem:sumconverge,theorem:xhatkyhatkConverge}.
 Although the convergence of $ f(x^{k},y^{k})  \to f(x^{*},y^{*}) $  is not provided 
 theoretically yet, it is shown numerically in 
 \cref{fig:subgradient_methods_toy_t_0.png},
 which motivates our future work on the convergence of $ f(x^{k},y^{k})  \to 
 f(x^{*},y^{*}) $ or $(x^{k},y^{k})  \to (x^{*},y^{*})$.
\begin{figure}[H]
\centering
\includegraphics[width=0.8\textwidth]{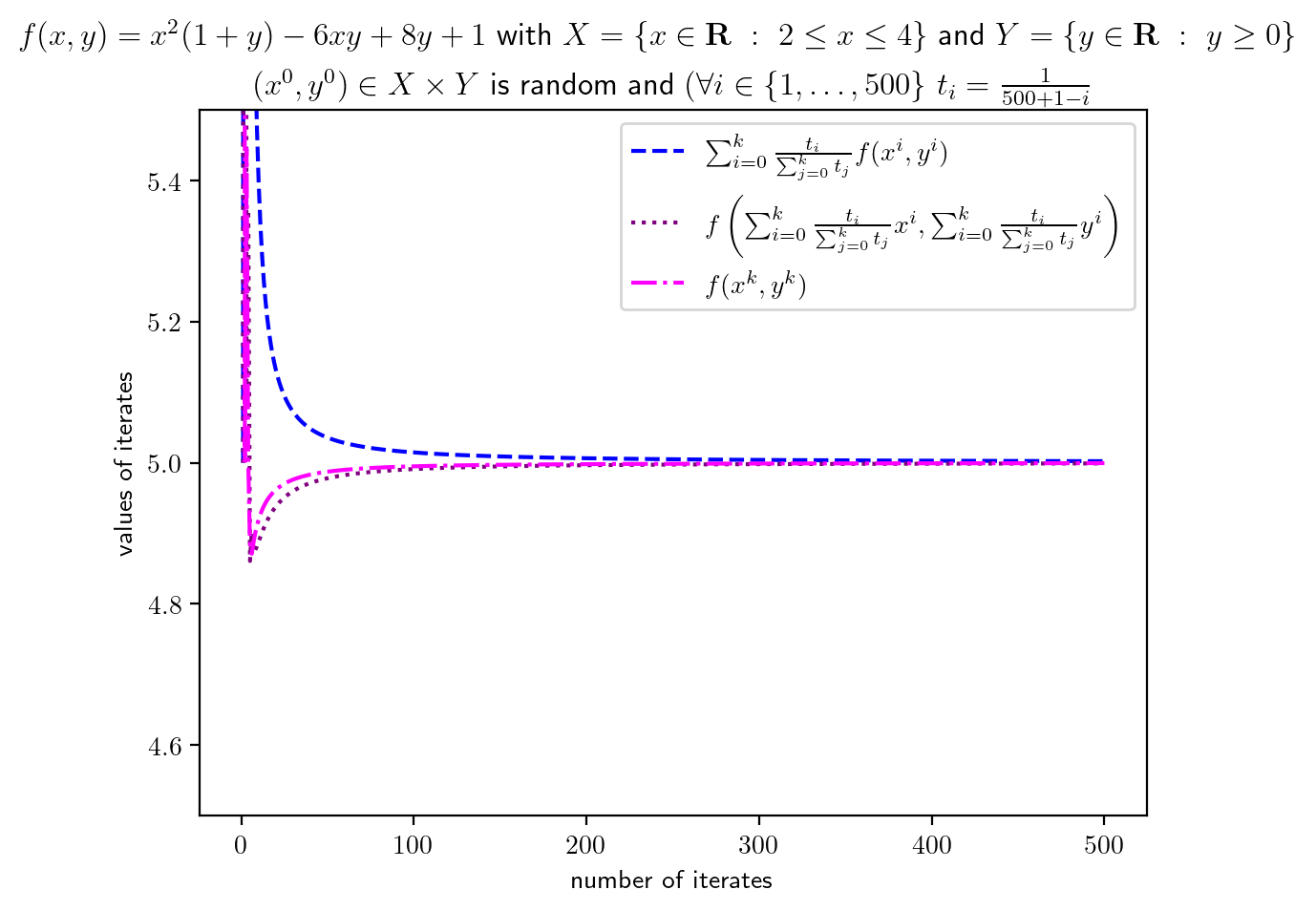}
\caption{Convergence result with randomly chosen initial points}
\label{fig:subgradient_methods_toy_t_0.png}
\end{figure}

\subsection{Linear program in inequality form} \label{subsection:LP}
Let $A \in \mathbf{R}^{m \times n}$,  $b  \in \mathbf{R}^{m}$,  and $c \in 
\mathbf{R}^{n}$.  
In this part, we consider the convex-concave function 
 $f : \mathbf{R}^{n}  \times \mathbf{R}^{m}_{+}  \to \mathbf{R}$ defined as 
\[ 
(\forall (x,y) \in \mathbf{R}^{n}  \times \mathbf{R}^{m}_{+} ) \quad	
f(x,y)  = y^{T}Ax +c^{T}x   -b^{T}y,
\] 
which is presented in \cref{example:lagrange}\cref{example:lagrange:lp}. 

 In our experiments, after randomly choosing $A \in \mathbf{R}^{100 \times 10}$, 
 $b  \in \mathbf{R}^{100}$, and $c \in \mathbf{R}^{10}$, we apply the 
 Python-embedded 
 modeling language CVXPY (see \cite{DiamondBoyd2016} for details) to find 
 bounded  and feasible problems. 
 Recall from \cref{example:lagrange}\cref{example:lagrange:lp} 
 that the convex-concave function $f$ above is the Lagrangian of a linear 
 programming with an inequality constraint and that the optimal solutions of the 
 related  primal and dual problem are both equal to the value of $f$ over a saddle-point. 
 We also solve the corresponding primal and dual problems by CVXPY
 to check the correctness of results from our algorithms. 
 
After finding problems with optimal solutions, we randomly choose initial points 
$(x^{0},y^{0}) \in \mathbf{R}^{10} \times \mathbf{R}^{100}$ and calculate  the 
sequence $\left(	\sum^{k}_{i =0}  \frac{t_{i}}{ \sum^{k}_{j =0}t_{j} }  f(x^{i}, y^{i})  
\right)_{1 \leq k \leq 100}$, $\left(  f(\hat{x}_{k},\hat{y}_{k}) \right)_{1 \leq k \leq 100}$, 
and $(  f(x^{k},y^{k}) )_{1 \leq k \leq 100}$ with  $(\forall k \in \{1, \ldots, 100\})$ $t_{k} 
=\frac{1}{100+1 -k}$.
We presented one result in \cref{fig:subgradient_methods_lp_random_0.png} 
below. 
Note that to see only the important range of y-axis,  we zoom in and set the y-axis 
limit on the picture as $[f(x^{*},y^{*})-1, f(x^{*},y^{*})+1]$,
where the optimal value
$f(x^{*},y^{*})$ is obtained from our  CVXPY code. 
Because in this case $f$ is linear, it's not a surprise  that $\left(	\sum^{k}_{i =0}  
\frac{t_{i}}{ \sum^{k}_{j =0}t_{j} }  f(x^{i}, y^{i})  
\right)_{1 \leq k \leq 100}$ and
$\left(  f(\hat{x}_{k},\hat{y}_{k}) \right)_{1 \leq k \leq 100}$  on 
\cref{fig:subgradient_methods_lp_random_0.png} are consistent.
 \begin{figure}[H]
\centering
\includegraphics[width=0.8\textwidth]{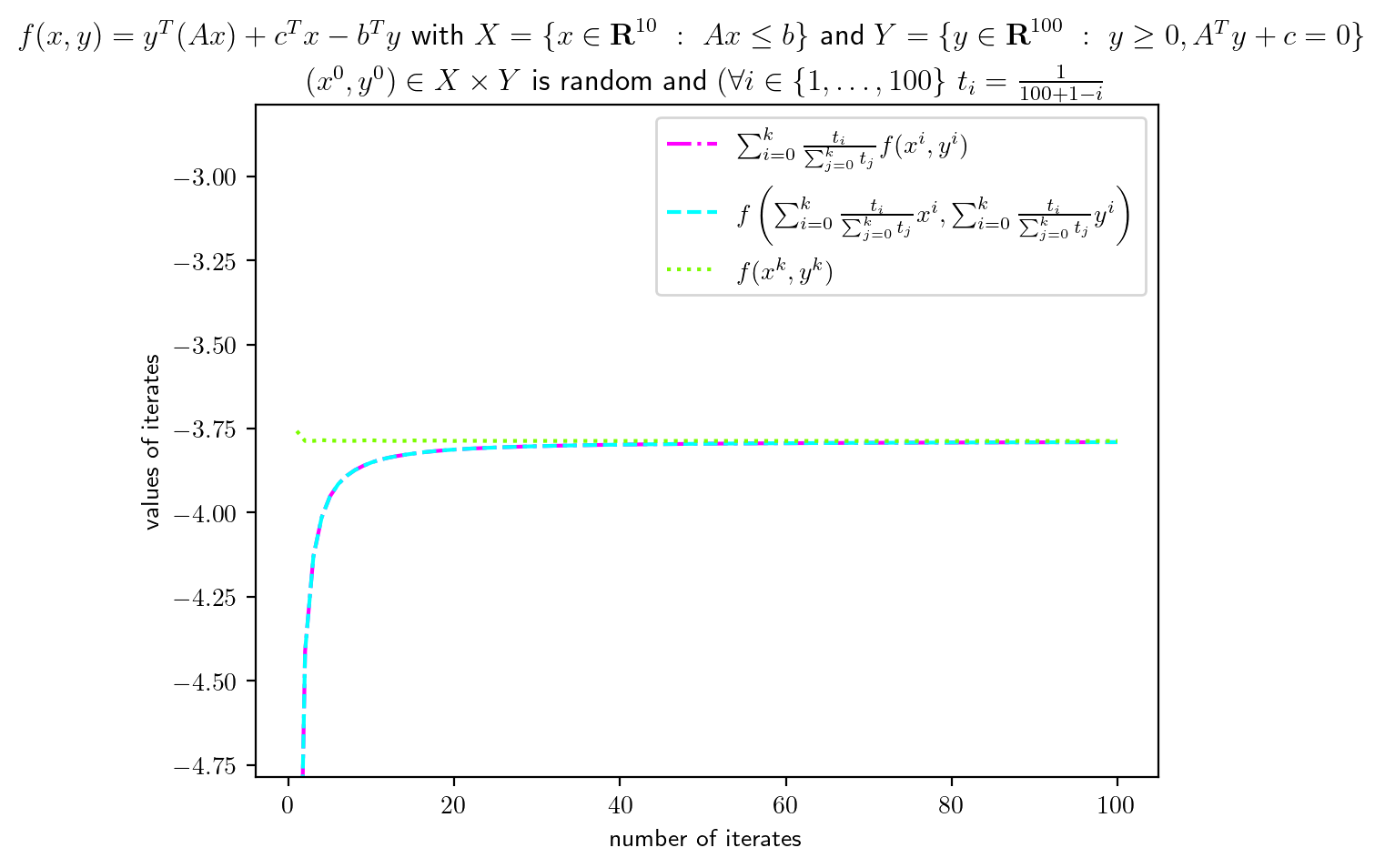}
\caption{Lagrangian of the inequality form LP}
\label{fig:subgradient_methods_lp_random_0.png}
\end{figure}
In all our experiments, the convergent point of our sequence of iterates
is identical with the optimal solutions obtained by CVXPY. 
The convergence of 
$\left(	\sum^{k}_{i =0}  \frac{t_{i}}{ \sum^{k}_{j =0}t_{j} }  f(x^{i}, y^{i})  \right)_{1 \leq k 
\leq 100}$ 
and
$\left(  f(\hat{x}_{k},\hat{y}_{k}) \right)_{1 \leq k \leq 100}$ confirms numerically our 
theoretical results in \cref{theorem:sumconverge,theorem:xhatkyhatkConverge}. 
Again, $(  f(x^{k},y^{k}) )_{1 \leq k \leq 100}$  converges to our required point 
although we have no theoretical support of the convergence yet. 

\subsection{Least-squares problem with \texorpdfstring{$\ell_{1}$}{$l_{1}$} 
regularization} 
\label{subsection:LSl1}
Let $A \in \mathbf{R}^{m \times n}$, $b  \in \mathbf{R}^{m}$, and $\gamma \in 
\mathbf{R}_{++}$. 
We consider the convex-concave function  
$f: \mathbf{R}^{n+m} \times \mathbf{R}^{m} \to \mathbf{R}$ defined as 
\[ 
(\forall ((x,u),y) \in \mathbf{R}^{n+m} \times \mathbf{R}^{m} )	
\quad  f((x,u),y) = \frac{1}{2} \norm{u}^{2}_{2} + \gamma \norm{x}_{1} + y^{T}( Ax -b -u),
\]
which is considered in \cref{example:lagrange}\cref{example:lagrange:lsl1} and 
\cref{example:lsl1:subgradient}.
As stated in \cref{example:lagrange}\cref{example:lagrange:lsl1}, 
$f$ is the Lagrangian of a least-squares problem with $\ell_{1}$ regularization. 
In our experiments, we set $\gamma =1$ and 
  randomly choosing $A \in \mathbf{R}^{100 \times 50}$ and $b \in 
\mathbf{R}^{100}$. Note that in this case the problem is always feasible and 
bounded. 
Similarly with \cref{subsection:LP}, 
we apply CVXPY  to solve related primal and dual problems. 

Consider $X =\mathbf{R}^{100+50}$ and $Y = \{ y \in \mathbf{R}^{100} ~:~ 
\norm{A^{T}y}_{\infty} \leq \gamma  \}$.  We randomly choose initial points 
$(x^{0},y^{0}) \in \mathbf{R}^{100+50}  \times \mathbf{R}^{100}$ and implement  
the sequences 
$\left(	\sum^{k}_{i =0}  \frac{t_{i}}{ \sum^{k}_{j =0}t_{j} }  f(x^{i}, y^{i})  
\right)_{1 \leq k \leq 500}$, 
$\left(  f(\hat{x}_{k},\hat{y}_{k}) \right)_{1 \leq k \leq 500}$, 
and 
$(  f(x^{k},y^{k}) )_{1 \leq k \leq 500}$ 
with  $(\forall k \in \{1, \ldots, 500\})$ $t_{k} 
=\frac{1}{500+1 -k}$.
We show one result in \cref{fig:subgradient_methods_lsl1_random_0.png} 
in which the convergent point is consistent with the optimal solution
obtained by CVXPY for corresponding primal and dual problems. 
With the optimal value
$f(x^{*},y^{*})$  obtained by our related CVXPY code, 
we zoom in and set the y-axis limit  on the picture as $[f(x^{*},y^{*})-1, 
f(x^{*},y^{*})+1]$ to see only the important range of y-axis.
It's interesting that $(  f(x^{k},y^{k}) )_{1 \leq k \leq 500}$ converges to 
the required optimal value  in this example as well. 
 
\begin{figure}[H]
\centering
\includegraphics[width=0.8\textwidth]{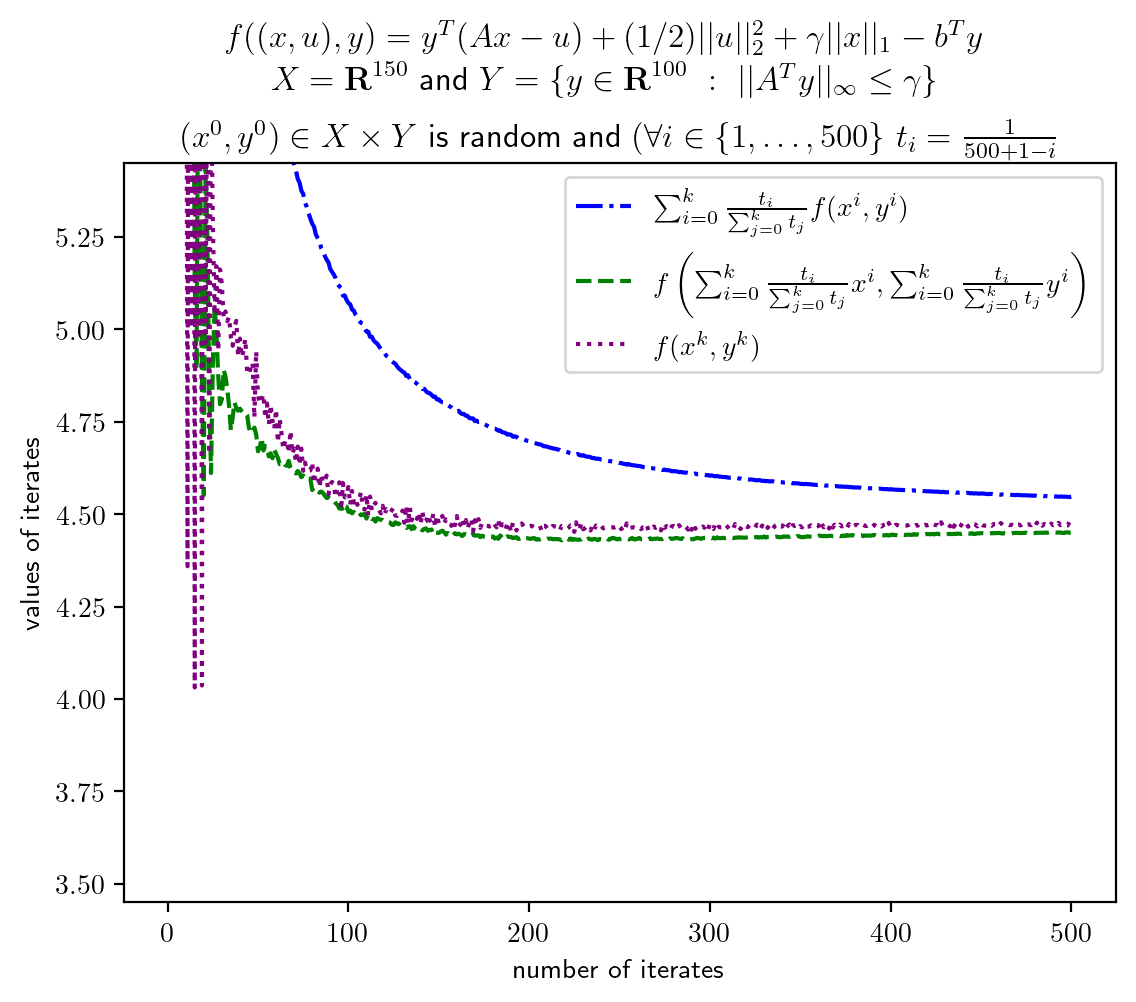}
\caption{Lagrangian of the least-squares problem with $\ell_{1}$ regularization}
\label{fig:subgradient_methods_lsl1_random_0.png}
\end{figure}

\subsection{Matrix game} \label{subsection:matrixgame}
We consider an easy version of the game interpretation of saddle-point problems 
below 
(see, e.g., \cite[Section~5.4.3]{BV2004} for details). 
Let   $C \in \mathbf{R}^{m \times 
n}$ 
and let $X \subseteq \mathbf{R}^{n}$ and $Y \subseteq \mathbf{R}^{m}$ be nonempty 
closed 
and convex subsets. 
The subject function in this case is $f: \mathbf{R}^{n} \times  \mathbf{R}^{m} 
\to 
\mathbf{R}$ defined as 
\[ 
	(\forall (x,y)\in X \times  Y) \quad f(x,y) = x^{T} Cy.
\]
In our experiment, we use the example of matrix game in 
\cite[Sections~5.2 and 5.3]{SLB2023} and set 
\begin{align*}
	&C = \begin{pmatrix}
		1 &2\\
		3&1
	\end{pmatrix},\\
&X:= \{ x \in \mathbf{R}^{2} ~:~ \sum^{2}_{i=1} x_{i} =1 \text{ and }(\forall i \in \{1,2\})~ 
x_{i} 
\geq 0  \}, \text{ and}\\
&Y:= \{ y \in \mathbf{R}^{2} ~:~ \sum^{2}_{i=1} y_{i} =1 \text{ and }(\forall i \in \{1,2\}) 
~y_{i} 
\geq 0  \}.
\end{align*}
In view of \cite[Section~5.3]{SLB2023}, the optimal value (that is the value of $f$ 
over the saddle-point) is $1.6667$ with keeping 4 decimal places.

In our experiments, we randomly 
choose initial points $(x^{0},y^{0}) \in X \times Y$ and calculate the sequences 
$\left(	
\sum^{k}_{i =0}  
\frac{t_{i}}{ \sum^{k}_{j =0}t_{j} }  f(x^{i}, y^{i})  
\right)_{1 \leq k \leq 1000}$, 
$\left(  f(\hat{x}_{k},\hat{y}_{k}) \right)_{1 \leq k \leq 1000}$, 
and 
$(  f(x^{k},y^{k}) )_{1 \leq k \leq 1000}$  with multiple choices of $(t_{k})_{k \in 
\mathbf{N}}$ but we noticed that in our experiments $\left(	\sum^{k}_{i =0}  
\frac{t_{i}}{ \sum^{k}_{j =0}t_{j} }  f(x^{i}, y^{i})  \right)_{1 \leq k \leq 1000}$ and 
$(  f(x^{k},y^{k}) )_{1 \leq k \leq 1000}$ converge very slow. 
To get the following \cref{fig:appmg_subgradient_methods_0.png}, we  randomly 
choose initial points $(x^{0},y^{0}) \in X \times Y$  and implement $\left(  
f(\hat{x}_{k},\hat{y}_{k}) \right)_{1 \leq k \leq 1000}$ with two cases of the parameter 
 $(t_{k})_{k \in \mathbf{N}}$: $(\forall i \in \mathbf{N})$ $t_{i}=0.1$ 
 and $(\forall i \in \mathbf{N})$ $t_{i}=0.01$. 
(We considered also other choices of $(t_{k})_{k \in \mathbf{N}}$ 
including $(\forall k \in \{1, \ldots, 1000\})$ $t_{k} =\frac{1}{1000+1 -k}$, but their 
convergence performances are not good. We calculated also $\left(	\sum^{k}_{i =0}  
\frac{t_{i}}{ \sum^{k}_{j =0}t_{j} }  f(x^{i}, y^{i})  \right)_{1 \leq k \leq 1000}$ and 
$(  f(x^{k},y^{k}) )_{1 \leq k \leq 1000}$ in the experiment associated with 
\cref{fig:appmg_subgradient_methods_0.png}, but they converge very slowly.) 
We observe that, in \cref{fig:appmg_subgradient_methods_0.png},  the convergent point 
is indeed the optimal solution obtained by 
\cite[Section~5.3]{SLB2023}.
\begin{figure}[H]
\centering
\includegraphics[width=0.6\textwidth]{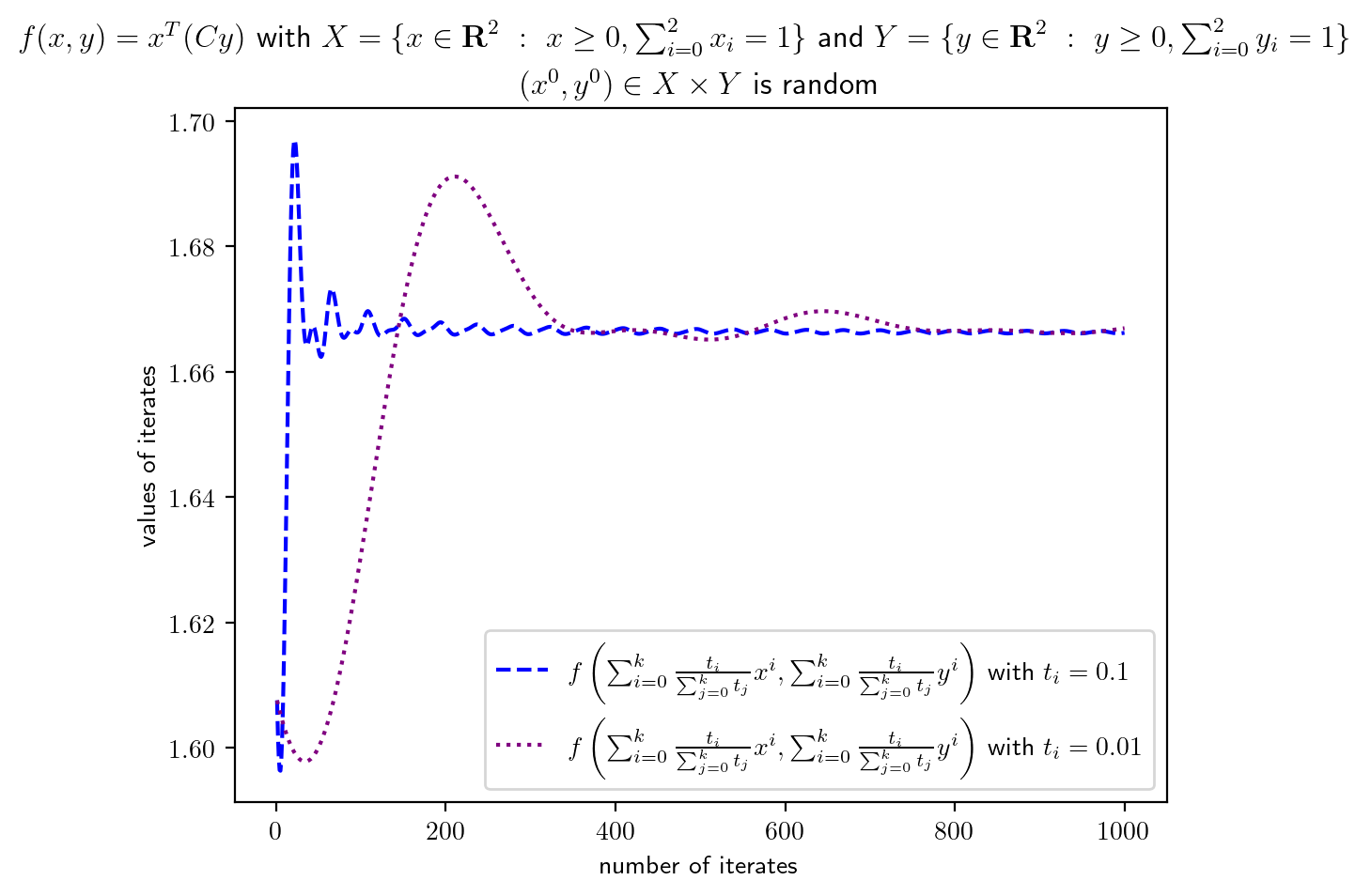}
\caption{Matrix game with  constant step sizes}
\label{fig:appmg_subgradient_methods_0.png}
\end{figure}
\subsection{Robust Markowitz portfolio construction problem}
\label{subsection:rmpcp}
As a last example in this section, we consider the robust 
Markowitz portfolio construction problem. 
(See, e.g., \cite{SLB2023} or \cite{ BBDKKNS2017} for details.)  
We consider the  robust Markowitz portfolio construction problem presented 
in  \cite[Sections~3.4 and 6.3]{SLB2023} here.
Let $\bar{\mu} \in \mathbf{R}^{n}$ be the  nominal mean, 
let $\bar{\Sigma} \in \symm^{n}_{++}$ be the nominal covariance, 
and  let $\gamma \in \mathbf{R}_{++}$. Let 
 $\mathcal{W} \subseteq \mathbf{R}^{n}$ be a convex set of feasible portfolios and let 
\begin{align*}
 \mathcal{M}  =& \{\bar{\mu}+\delta ~:~ (\forall i \in \{1,2,\ldots, n\}) \abs{\delta_{i}} 
 \leq \rho_{i}\},\\
\mathcal{S}  =& \{\bar{\Sigma}+\Delta ~:~ \bar{\Sigma}+\Delta \in \symm^{n}_{+} 
\text{ and } (\forall i,j \in \{1,2,\ldots, n\})    \abs{\Delta_{ij}} \leq \eta(\bar{\Sigma}_{ii}  
\bar{\Sigma}_{jj}   )^{\tfrac{1}{2}}\},
\end{align*}
  where  $\rho \in \mathbf{R}^{n}_{++}$ is a vector of uncertainties in the forecast 
  returns, 
  $\delta \in \mathbf{R}^{n}$ is the perturbation of the nominal mean $\bar{\mu}$,  
$\Delta \in \symm^{n}$ is the perturbation of the nominal covariance 
$\bar{\Sigma}$, and $\eta \in (0,1)$ is a parameter scaling the perturbation 
to the forecast covariance matrix. 

In this case, the convex-concave function is $f: \mathcal{M} \times 
\mathcal{S}  
\times \mathcal{W}   \to \mathbf{R}$ defined as 
\[ 
(\forall (x,y)= ((\mu,\Sigma),w) \in ( \mathcal{M} \times \mathcal{S} ) \times \mathcal{W}  )	
\quad  f(x,y)=f((\mu,\Sigma),w) = \mu^{T}w -\gamma w^{T}\Sigma w.
\]

We also use the data ($\bar{\mu}$, $\bar{\Sigma}$, $\rho$, $\eta$, and $\gamma$) 
worked in \cite[Sections~6.3]{SLB2023}  where $n=6$. In view of the result 
obtained therein, the optimal value 
(that is, the value $f((\mu^{*},\Sigma^{*}),w^{*}) $ over the saddle-point) is $0.076$.

We randomly choose $x^{0}= (\mu^{0},\Sigma^{0} ) \in  \mathcal{M} \times 
\mathcal{S} $ and $ y^{0}=w^{0} \in \mathcal{W} $ and calculate the 
sequences 
$\left(	\sum^{k}_{i =0}  \frac{t_{i}}{ \sum^{k}_{j =0}t_{j} }  f(x^{i}, y^{i})  \right)_{1 \leq k 
	\leq 1000}$
and $(  f(x^{k},y^{k}) )_{1 \leq k \leq 1000}$ 
with  $(\forall k \in \{1, \ldots, 1000\})$ $t_{k} =\frac{1}{1000+1 -k}$.
 We present one result in \cref{fig:app_rmpc_0.png} and see clearly 
 that $\left(	\sum^{k}_{i =0}  \frac{t_{i}}{ \sum^{k}_{j =0}t_{j} }  f(x^{i}, y^{i})  \right)_{1 
 \leq k 
 	\leq 1000}$ 
 and  $\left(  f(x_{k}, y_{k}) \right)_{k \in \mathbf{N}}$ converge to the required 
 optimal value $0.076$ in $1000$ iterates.
 The performance of 
 $\left(f(\hat{x}_{k},\hat{y}_{k}) \right)_{k \in \mathbf{N}}$ 
 in our experiments of this example is not good
 so we don't present it in \cref{fig:app_rmpc_0.png} below.
 
\begin{figure}[H]
\centering
\includegraphics[width=0.6\textwidth]{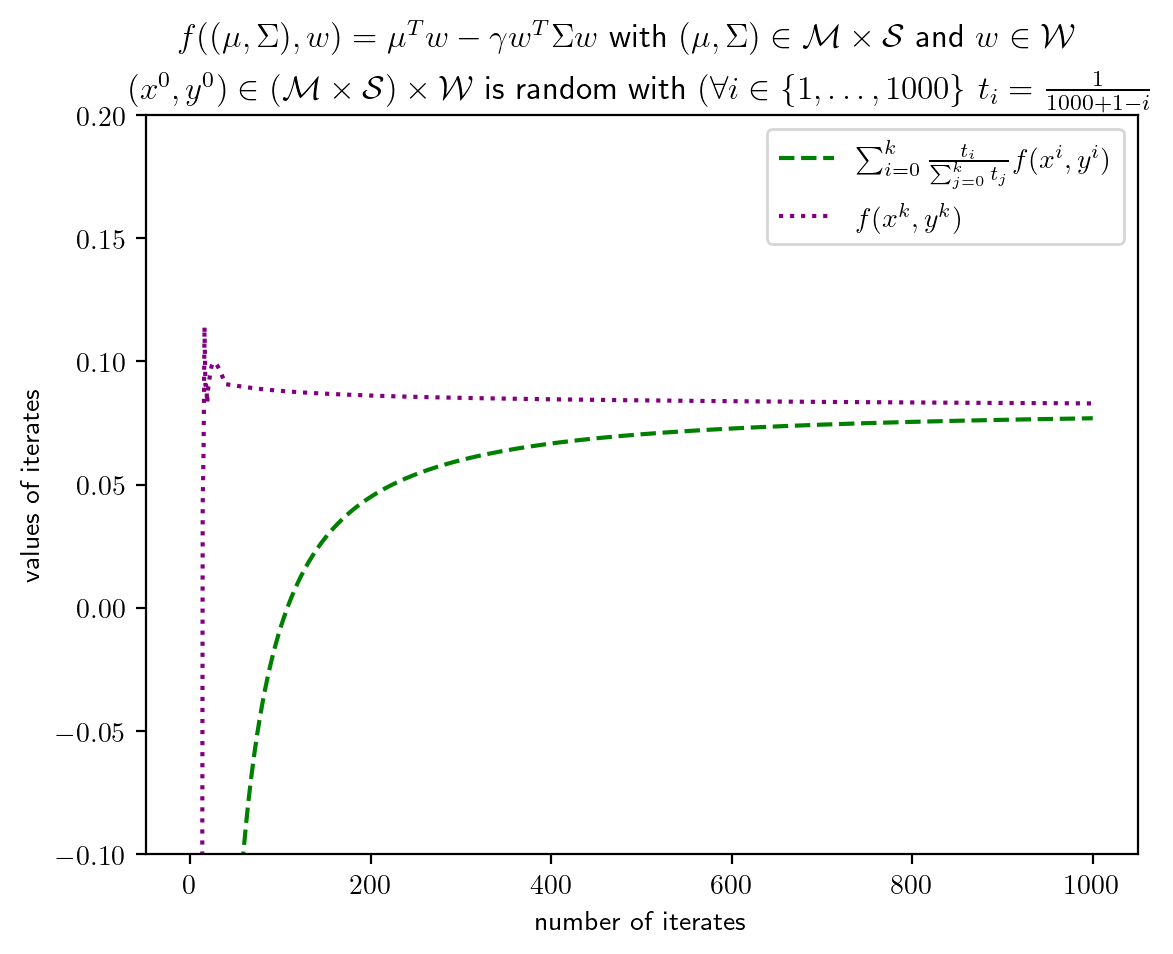}
\caption{Convergence of robust Markowitz portfolio construction problem}
\label{fig:app_rmpc_0.png}
\end{figure}

 \section{Conclusion} \label{section:conclusion}
In this work,  
we proved some convergence results  of our alternating subgradient method 
for  convex-concave saddle-point problems associated with
general convex-concave functions. 
Let $(x^{*}, y^{*}) \in X \times Y$ be a saddle-point of 
a convex-concave function $f$, 
and let $((x^{k}, y^{k}))_{k \in \mathbf{N}}$ be 
a sequence of iterates 
generated by  our alternating subgradient method 
associated with the sequence $(t_{k})_{k \in \mathbf{N}}$ of step sizes.
We presented the convergence 
$	\sum^{k}_{i =0} \frac{t_{i}}{ \sum^{k}_{j=0} t_{j}} f(x^{i}, y^{i}) \to  f(x^{*}, y^{*}) $ 
under some popular assumptions on the step-size. 
With the same assumption, we also showed 
$f\left( \sum^{k}_{i =0} 
\frac{t_{i}}{ \sum^{k}_{j=0} t_{j}} x^{i} ,
 \sum^{k}_{i =0} \frac{t_{i}}{ \sum^{k}_{j=0} t_{j}} y^{i}\right) \to  f(x^{*}, y^{*})$.

Our convergence results were confirmed 
by our numerical experiments in examples of 
a linear program in inequality form, 
a least-squares problem with $\ell_{1}$ regularization,
a matrix game, and a robust Markowitz portfolio construction problem.
We also compared our iterate scheme (associated with a not summable 
but square summable sequence $(t_{k})_{k \in \mathbf{N}}$  of step sizes)
with iterate schemes associated with constant step sizes on a Lagrangian of 
an easy convex constrained optimization problem. 
Our numerical result showed  some benefits of
replacing constant step sizes with our step sizes $(t_{k})_{k \in \mathbf{N}}$. 
Additionally, in our numerical experiments, we  displayed 
the convergence  $f(x^{k}, y^{k}) \to f(x^{*}, y^{*})$ in multiple examples, 
which currently lacks theoretical support, to the best of our knowledge.  
This motivates our future work on theoretical proof of the convergence of 
$(x^{k}, y^{k}) \to (x^{*}, y^{*})$ and $f(x^{k}, y^{k}) \to f(x^{*}, y^{*})$.
 
  \section*{Acknowledgments}
  Hui Ouyang acknowledges the support of
  the Natural Sciences and Engineering Research Council of Canada (NSERC), 
  [funding reference number PDF – 567644 – 2022]. 
  
 \addcontentsline{toc}{section}{References}
 \clearpage
 \bibliographystyle{abbrv}
 \bibliography{ccspp_subgradient}
 \clearpage
\end{document}